\documentclass{amsart} 
\usepackage{mathmiosec}
\usepackage{tikz-cd} 
\renewcommand{\#}{\sharp}

\def\L{\mathcal{L}} 
\def\ra{\rightarrow} 
\def\N{\mathbb{N}}
 
\def\slee{{\em s}-Lee\ }
\DeclareMathOperator{\Aff}{Aff}

\title{Locally conformally symplectic convexity} 
\author{F.~Belgun}\thanks{The first named author was supported by a grant of Ministry of Research and Innovation, CNCS - UEFISCDI, project number PN-III-P4-ID-PCE-2016-0065, within PNCDI III.}
\author{O.~Goertsches}
\author{D.~Petrecca}
\address{(F.~Belgun) Institute of Mathematics of the
Romanian Academy - 21 Calea Grivitei Street, 010702 Bucharest,
Romania.}  \email{florin.belgun at math.uni-hamburg.de}
\address{(O.~Goertsches) Fachbereich Mathematik und Informatik der
Philipps-Universit\"at Marburg - Hans-Meerwein-Strasse 6, Marburg,
Germany.}  \email{goertsch at mathematik.uni-marburg.de}
\address{(D.~Petrecca) Institut f\"ur Differentialgeometrie, Leibniz
Universit\"at Hannover - Welfengarten 1, Hannover, Germany.  }
\email{petrecca at math.uni-hannover.de}

\begin{document}

\begin{abstract} We investigate special lcs and twisted Hamiltonian torus actions on strict lcs manifolds and characterize them geometrically in terms of the minimal presentation. We prove a convexity theorem for the corresponding twisted moment map, establishing thus an analog of the symplectic convexity theorem of Atiyah and Guillemin-Sternberg. We also prove similar results for the symplectic moment map (defined on the minimal presentation) whose image is then a convex cone. In the special case of a compact toric Vaisman manifold, we obtain a structure theorem.
\end{abstract}
\maketitle
\section{Introduction} In the 1980's, Atiyah and Guillemin--Sternberg
independently proved the celebrated convexity theorem in symplectic
geometry, see~\cite{atiyah_conv, techniques}. It states that if a
compact torus $T$ acts in a Hamiltonian fashion on a compact
symplectic manifold $(M, \omega)$ with moment map $\mu \colon M \to
\lie t^*$, then the image $\Delta = \mu(M) \subset \lie t^*$ is a
convex polytope given by the the convex hull of the points in $M$
fixed by $T$, and the fiber $\mu^{-1}(\alpha)$ of every $\alpha \in
\Delta$ is connected.

The approach followed in Atiyah's proof is based on the fact that a
generic component of the moment map is a Morse-Bott function on
$M$. An alternative technique used for the proof of the above theorem
makes use of the \emph{local-global principle}, that relates the
openness and convexity of a vector-valued map on a topological space
to its \emph{local} openness and convexity, under relatively mild
global topological assumptions (properness), see
\cite{HilgertNeebPlank} or \cite{BjoKar}. This principle is the link
between the local form of a moment map and its global convexity.



The symplectic convexity theorem has variants in a multitude of other
geometries. They range from odd-dimensional analogs of symplectic
geometry, as contact \cite{lerman} or cosymplectic geometry
\cite{BazzoniGoertsches_convexity} to generalizations of symplectic
structures, as $b^m$-symplectic \cite{convexity_bm} or generalized
$H$-twisted complex geometry~\cite{NittaGeneralized}.

We focus on compact \emph{locally conformally symplectic (lcs)}
manifolds. They are compact manifolds together with a non-degenerate
$2$-form $\omega$ that, around every point, is conformal to a symplectic
form. Their study was initiated in~\cite{Vaisman_lcs}. Often one
investigates the important special cases of \emph{locally conformally
  K\"ahler manifolds (lcK)} and \emph{Vaisman manifolds}, where the $2$-form $\omega$ is the fundamental form of a Hermitian metric, which in the Vaisman case also admits a non-zero parallel vector field.

The kind of
transformations that we consider are conformal with respect to this
structure and satisfy a twisted Hamiltonian condition, making possible
to consider a \emph{twisted moment map}. We refer to
Section~\ref{sec:prelim} for an introduction to the content
needed for this paper (see also \cite{DO} for further details).

The study of transformations on lcs manifolds was considered by Haller
and Rybicki~\cite{HallerRybicki}, where they study lcs reduction, and
recently specialized to lcK manifolds by Istrati~\cite{istrati_toric}
and to Vaisman manifolds by Pilca~\cite{Pilca}.

In this paper we prove a convexity theorem for twisted Hamiltonian
actions on locally conformally symplectic manifolds, see Theorem
\ref{thm:convmu} and Proposition \ref{prop:imagepolytope}:
\begin{thm} \label{thm:convexIntro} Consider a twisted Hamiltonian
action of Lee type of a compact torus $T$ on a compact lcs manifold
$(M, \omega, \theta)$, with twisted moment map $\mu \colon M
\rightarrow \lie t^*$. Then $\mu:M\rightarrow \mu(M)$ is an open map
with connected level sets, whose image is a convex polytope.
\end{thm}

While the condition of being twisted Hamiltonian is invariant under
conformal changes, the convexity of the moment image is not, see
Example~\ref{ex:momentimageconfchange}. This is reflected in our
theorem by the fact that the Lee type condition, i.e. the condition
that the anti-Lee (or {\em s}-Lee) vector field is a fundamental vector
field of the action, is not conformally invariant either. We consider, however, also the {\em symplectic moment map} (see below) which is a conformal invariant.

\medskip

The paper is organized as follows: In Section \ref{sec:prelim} we review the necessary notions of lcs geometry, including the minimal presentation $\hat M$ of an lcs manifold $M$, and the Vaisman-Sasaki correspondence, and also give a topological characterization of the rank one case (Proposition \ref{rk1}, generalized to Proposition \ref{Fprop}).

In Section \ref{sec:3} we introduce the lcs, special lcs and twisted Hamiltonian vector fields. The difference between the latter two kinds is measured by the {\em Morse-Novikov cohomology}, which is, however, difficult to compute. We provide criteria characterizing special lcs vector fields which are twisted Hamiltonian (Proposition \ref{affine}, Proposition \ref{dexact} and their corollaries).

In Section \ref{sec:4} we discuss Lie group actions on lcs manifolds and show three essential results for the sequel: in Proposition \ref{closure-sLee} we show that if the Lee vector field preserves infinitesimally some metric, then it generates a special lcs action of a torus, in Lemma \ref{liftG} we show that a special lcs action of $G$ on $M$ lifts to a symplectic action of $G$ on the minimal covering $\hat M$, and, in Propositions \ref{equivmmap} and \ref{Tequiv}, we show that a twisted Hamiltonian action of Lee type of a torus has an equivariant twisted moment map.

In Section \ref{sec:5}, we prove our main results (Theorem \ref{thm:convexIntro}) and, moreover, discuss another object defined by a twisted Hamiltonian action: the {\em symplectic moment map}, corresponding to the symplectic (in fact, Hamiltonian) action of $G$ on the minimal covering $\hat M$. The advantage of $\hat \mu$ is that it does not depend on any choice of an lcs form $\omega$, but only on its conformal class (thus on the lcs structure).

Investigating the convexity of $\hat\mu$ directly is, however, more difficult, since this map is usually not proper (see Lemma \ref{hat-proper}) (recall that properness is an essential ingredient for applying the local-global principle).

However, we can prove that properness, thus also the convexity of the symplectic moment map holds for torus actions on lcs manifolds of rank one (Theorem \ref{convx-hat}), under some extra assumption.

For higher rank, we were able to prove, also under some extra assumptions, that the image of the symplectic moment map is a cone over the image of the twisted moment map (Theorem \ref{conemoment}), and is, therefore, convex.


In Section~\ref{sec:momentmapvaisman} we specialize to the case of a
Vaisman manifold and relate the moment image to that of the associated
Sasakian manifold and to the one of the \emph{minimal K\"ahler
covering} (Proposition~\ref{prop:momentimageSasakian}).

In Section~\ref{subsec:toric}, we provide a structure theorem
(Theorem~\ref{thm:prod}) for toric Vaisman manifolds. Such a manifold has been already proved in ~\cite{MaMoPi} to be a mapping torus of a toric Sasakian manifold. Using toric geometry on K\"ahler cones, we are able to prove, moreover, that the
transformation defining the mapping torus actually lies in the acting
torus. This leads to the consequence that every toric Vaisman manifold
is diffeomorphic, as a smooth manifold, to the product of a toric
Sasakian manifold with a circle. 

Finally, we provide in Section~\ref{sec:examples} examples of twisted
Hamiltonian actions on lcs, lcK and Vaisman manifolds and compute their
moment images.



\subsection*{Acknowledgments} The authors would like to thank Liviu
Ornea, Alexandra Otiman and Mihaela Pilca for the fruitful discussions
and their useful remarks.

\section{Preliminaries on lcs, lcK and Vaisman manifolds}
\label{sec:prelim}

Let $(M, \omega)$ be an almost symplectic manifold of real dimension
greater than 2, where $\omega$ is a non-degenerate 2-form. Often
$\omega:=g(J\cdot,\cdot)$ will be the fundamental form of an (almost)
Hermitian metric $g$ on $(M,J)$, where $J:TM\rightarrow TM$ is an
(almost) complex structure on $M$. We will usually consider the
complex case, when $J$ is integrable.

If every point of $M$ admits a neighborhood $U$ and a smooth function
$f_U \colon U \to \R$ such that the two-form $e^{-f_U} \omega|_U$ is
closed, we call $(M, \omega)$ a \emph{locally conformally symplectic
manifold (lcs)}. If $\omega$ is the fundamental (or K\"ahler) form of
a Hermitian manifold $(M, g, J)$, then we call it  a \emph{locally
conformally K\"ahler manifold (lcK)}. 

From the definition, it follows that the local $1$-forms $df_U$ glue
together to a global $1$-form $\theta$, called the \emph{Lee form},
satisfying on $M$
\begin{equation} \label{eq:lcs} d \omega = \theta \wedge \omega.
\end{equation}

Thus, the 2-form $\omega$ is closed with respect to the so-called
\emph{twisted differential} $d_\theta: = d - \theta \wedge \cdot$. 
Because $\theta$ is closed, we have $d_\theta^2=0$, thus the twisted
differential defines a cohomology (called the {\em Morse-Novikov}
cohomology) $H^*(M,d_\theta)$.
 
In general, the local functions $f_U$ do not glue together to a global
function $f$ on~$M$. When this happens, or equivalently if $\theta$ is
exact, the structure is called \emph{globally conformally symplectic
(gcs)}, resp. \emph{globally conformally K\"ahler (gcK)}, as this
means that $e^{-f} \omega$ is a symplectic form, resp. $e^{-f} g$ is a
K\"ahler metric on $M$. If a connected manifold is gcs, resp. gcK,
then all symplectic, resp. K\"ahler structures in the given conformal
class are homothetic to each other.
  
 We call an lcs structure \emph{strict} if it is not gcs, and an lcK
structure \emph{strict} if it is not gcK. Note that the both the lcs
and the lcK condition are invariant under conformal changes
(multiplication by a positive function) of $\omega$, resp. $g$. In
this paper, unless explicitly mentioned, all the considered lcs, resp
lcK structures will be strict.

 According to the classical result of P. Gauduchon \cite{gaud}, one
can always find a metric (called {\em Gauduchon metric}) in a
conformal lcK class, such that the corresponding Lee form is harmonic
(w.r.t. that metric itself). Such Gauduchon metrics are homothetic to
each other, hence a connected group that preserves the conformal and
the complex structure of an lcK manifold $(M,J,\omega)$ acts by
isometries of any Gauduchon metric (see also \cite[Proposition
3.8]{MaMoPi}). 
 
If, for a strict lcK manifold $M$, the form $\theta$ is parallel with
respect to the metric $g$, then  $M$ is called a \emph{Vaisman
manifold}. For simplicity, we will make the convention that the Lee
form of a Vaisman manifold has unit length.  \medskip

The \emph{anti-Lee} vector field of an lcs manifold
$(M,\omega,\theta)$ is the unique vector field $V$ such that 
\begin{equation}\label{aLlcs}  \omega(V,\cdot)=-\theta.
\end{equation} In the lcK setting it is $V=J\theta^\#$, where
$\theta^\#$, the {\emph{Lee vector field}}, is dual to the Lee form
$\theta$ with respect to the metric $g$. As there is no Lee vector
field in the lcs setting, we propose to call $V$ the {\em \slee vector
field} (the prefix ``s'' coming from {\em symplectic}) in the most
general setting.

\subsection{Presentations of lcs and lcK
manifolds} \label{sec:presentations} We characterize now lcs,
resp. lcK structures by symplectic, resp. K\"ahler structures on a covering.
\begin{defi}\label{present}
A {\em presentation} $(M_0, \omega_0, \Gamma, \rho, \lambda)$ 
of an lcs structure on $M$ consists of a symplectic manifold $(M_0,
\omega_0)$,  a discrete group $\Gamma$ acting freely and properly on $M_0$ such that $M=M_0/\Gamma$ and
\begin{equation}\label{eq:pres-hom}\gamma^*\omega_0=e^{-\rho(\gamma)}\omega_0\end{equation} 
for all $\gamma\in \Gamma$, where $\rho:\Gamma\ra\R$ is a group homomorphism, and
$\lambda:M_0\ra\R$ is a $\rho$-equivariant smooth function, i.e.
\begin{equation}\label{eq:pres}
\lambda\circ\gamma-\lambda=\rho(\gamma)
\end{equation}
for all $\gamma\in \Gamma$.
\end{defi}
One says that \eqref{eq:pres} means that $\lambda$ is $\rho$-equivariant,
and \eqref{eq:pres-hom} that $\Gamma$ acts on
$M_0$ by symplectic homotheties. 
\smallskip

The correspondence between the lcs structure $(M,\omega)$ and the
above presentation is the following: first, $M= M_0/\Gamma$, and $p:M_0\ra M$ is a Galois covering. Then,
\begin{equation}\label{pres-lcs}
  p^*\omega=e^\lambda\omega_0,\ p^*\theta=d\lambda.\end{equation}
More precisely, given a presentation as above, the right hand sides of
the above equalities are $\Gamma$-invariant forms on $M_0$ and define
thus $\omega$ and $\theta$ on $M=M_0/\Gamma$.
\smallskip

Conversely, given an lcs structure on $M$, and a Galois covering
$p:M_0\ra M$ with group $\Gamma$ such that $p^*\theta$ is exact, we define $\lambda$ as
a primitive of $p^*\theta$. Then the group homomorphism $\rho$ is defined by the relation \eqref{eq:pres} (note that the right hand side has to be a
constant depending on $\gamma$). Setting
\begin{equation}\label{eq:omega0}
\omega_0 := e^{-\lambda} p^*\omega,
\end{equation}
we clearly get a symplectic form $\omega_0$ and the equation
\eqref{eq:pres-hom}. \smallskip

If, moreover, $M$ has an lcK structure with complex structure $J$, then
we consider on $M_0$ the lifted complex
structure and we obtain that $(\omega_0, J)$ is a K\"ahler structure
and $\Gamma$ is a discrete group of holomorphic
homotheties. Conversely, if a presentation $(M_0,\omega_0,\Gamma,\rho,\lambda)$
as above admits a $\Gamma$-invariant complex structure $J$ on
$M_0$, then the quotient space $(M,J)$ inherits an lcK metric
$g:=\omega(\cdot,J\cdot)$. Note that for the K\"ahler metric
$g_0:=\omega_0(\cdot,J\cdot)$, $\Gamma$ acts, in particular, by (metric)
homotheties. 


\begin{rem}  Note that, if $M$ is a {\em strict} lcs or lcK
manifold, then the group homomorphism $\rho$ is non-trivial, thus
$\Gamma$ must be infinite.\end{rem} 


\begin{example}If $(M, \omega, \theta)$ is an lcs (resp.\ lcK)
manifold, let $\pi \colon \tilde M \to M$ be the universal covering
$M$ with the pull-back structure. Let $\lambda$ be a primitive of
$\pi^* \theta$. The fundamental group of $M$ leaves $\pi^* \theta = d
\lambda$ invariant, so there is a homomorphism $\tilde \rho \colon
\pi_1(M) \to \R$ that maps $\gamma \in \pi_1(M)$ to the number
$c_\gamma$ such that $\gamma \circ \lambda = \lambda + c_\gamma$.
\end{example}

\begin{example}\label{homcov} Let $M$ be a connected manifold. Consider 
\[
\Gamma:=\pi_1(M)/[\pi_1(M),\pi_1(M)]\simeq H_1(M,\Z)
\] to be quotient of of the fundamental group by its commutator $[\pi_1(M),\pi_1(M)]$, the result being the integer homology group. Then $\Gamma$ is the (Abelian) group of a Galois covering $\bar p:\bar M\ra M$, called the {\em homological} covering $\bar M:=\tilde M/[\pi_1(M),\pi_1(M)]$. As above, if $(M,\omega,\theta)$ is lcs, then we get a presentation on the homological covering $\bar M$ because all the objects that we define on the universal covering are $[\pi_1(M),\pi_1(M)]$-invariant. Note that, like the universal covering, the homological covering has the property that the pull-back through $\bar p$ of any closed form on $M$ is exact (on $\bar M$).\end{example}

For us, the main example of a presentation, besides the universal and the homological 
covering, is the {\em minimal covering} $\hat M:=\tilde M/\Gamma_0$,
where $\Gamma_0:=\ker\tilde\rho$ is the (normal) subgroup of
$\pi_1(M)$ that consists of the elements that act on $(\tilde
M,\omega_0)$ by holomorphic isometries (in the lcs case only symplectomorphisms). The corresponding group $\hat\Gamma:=\pi_1(M)/\Gamma_0$ is then a free Abelian $\Z^k$, because
the induced group homomorphism $\hat\rho:\hat\Gamma\ra \R$ is
injective. The number $k$ is called the {\em lcK rank} (or the {\em
  lcs rank}) of $(M,\omega)$ and it is a number between $1$ and
$b_1(M)$, see also \cite[Prop.~2.12]{GOOP}.
\begin{rem}\label{lcsrk} The lcs (lcK) rank of $(M,\omega)$ is the rank
  of the free abelian subgroup  $\rho(\Gamma)\subset \R$, which is the
  same for all presentations $(M_0,\omega_0,\Gamma,\rho,\lambda)$.\end{rem}
\begin{rem} The minimal covering, and thus the lcK or lcs rank defined
above depend only on the closed, non-exact 1-form $\theta$ (more
precisely on its cohomology class in $H^1(M,\R)$): $M_0$ is the
quotient of $\tilde M$ by the subgroup of $\pi_1(M)$ that preserves a
primitive $\lambda$ of $\theta$ (hence any).\end{rem}

The case where the lcs (lcK) rank is $1$ is
  special from a topological viewpoint.
  \begin{lemma}\label{circl} If $(M,\omega)$ has lcs rank $1$, then
  there is a map $\bar\lambda:M\ra S^1$ such that, if we denote by $q$
  the projection from $\R$ to $\R/\rho(\Gamma)\simeq S^1$, and by
  $\hat p:\hat M\ra M$ the projection of the minimal covering, we have
$$q\circ \lambda=\bar\lambda\circ\hat p.$$
\end{lemma}
\begin{proof}From Remark \ref{lcsrk},  $\Gamma\simeq\Z$. The lemma follows now from the $\rho$-equivariance of $\lambda$.\end{proof}
As a consequence, we prove the following proposition.
\begin{prop}\label{rk1} Let $(M,\omega)$ be an lcs manifold.
The primitive $\lambda:M_0\ra\R$ of the pull-back of the Lee form
  $\theta$ to a presentation $M_0$ is a proper map if and only  if
  the lcs rank is $1$, $M$ is compact and $M_0$ is (a finite covering
  of) the minimal presentation of $M$. 
\end{prop}
\begin{proof} Assume that $\lambda:M_0\ra\R$ is proper. We first show that the lcs rank of $M$ is $1$, i.e., by Remark \ref{lcsrk}, that $\rho(\Gamma)$ is a cyclic subgroup of $\R$. By \eqref{eq:pres} this is true if $\lambda(\Gamma\cdot x)\subset \R$ is closed for some arbitrary fixed $x\in M_0$. So suppose there is a sequence $\gamma_n\cdot x$, with $\gamma_n\in \Gamma$, such that $\lambda(\gamma_n\cdot x)$ is a Cauchy sequence. Let $A$ be the closure of
$\{\lambda(\gamma_n\cdot x)\ |\ n\in\N\}$ in $\R$. Then $A$ is compact, and 
thus, using the properness of $\lambda$, the preimage $\lambda^{-1}(A)\subset M_0$
is compact as well. 

The properness of the free action of the discrete group $\Gamma$ on $M_0$ implies that the sequence $(\gamma_n)_{n\in \N}$ contains only finitely many elements of $\Gamma$. Thus the set $\{\gamma_n\cdot x\ |\ n\in\N\}$ is finite and so is $A$. This implies that $\lambda(\Gamma\cdot x)\subset\R$ is closed, and hence the lcs rank of $M$ is $1$.

Taking the above sequence $(\gamma_n)_{n\in \N}$ in the normal subgroup
$\Gamma_0:=\ker\rho\subset \Gamma$ we obtain using \eqref{eq:pres} and the properness of $\lambda$ that $\Gamma_0$ is finite, and thus $M_0$ is a finite covering of the minimal covering $\hat M\simeq M_0/\Gamma_0$. 
 
To prove compactness of $M$, let $(x_n)_{n\in \N}$ be a sequence in
$M$ and let $(y_n)_{n\in\N}$ be a sequence in $M_0$ such that
$p(y_n)=x_n$ for all $n\in\N$. Let $a>0$ be a generator of
$\rho(\Gamma)\subset\R$. By shifting
the points $y_n\in M_0$ by elements of $\Gamma$, we can suppose that $\lambda(y_n)$ is contained in the interval $[0,a]$ for all $n \in \N$. By the properness of $\lambda$, $(y_n)$ is a sequence in the compact set $\lambda^{-1}([0,a])$. Therefore
$(y_n)_{n\in\N}$ and thus also $(x_n)_{n\in\N}$ admit converging
subsequences.
\smallskip

Let us now show the converse direction. We use the notation from Lemma \ref{circl}. 
For a compact interval $I$ of length less than $a$, the positive generator of the cyclic group $\rho(\Gamma)\subset\R$, $\lambda^{-1}(I)$ is homeomorphically
projected by $\hat p$ onto $\bar\lambda^{-1}(q(I))$. This latter set
is closed in $M$, thus compact. As any compact set in $\R$ can be covered by finitely many sets of type $I$, we conclude that $\lambda:M_0\ra\R$ is proper. 
\end{proof}
We formulate here a very general topological fact that generalizes both Lemma \ref{circl} and Proposition  \ref{rk1}.
\begin{prop}\label{Fprop} Let $\bar p:\bar M\ra M$ be the homological
covering of a compact connected manifold $M$.

Let
$\alpha\in C^\infty(\Lambda^1M\otimes H_1(M,\R))$ be a closed form
with values in $H_1(M,\R)\simeq H^1(M,\R)^*$, representing the
tautological class $Id\in H^1(M,\R)\otimes H_1(M,\R)$. Then the
pull-back of $\alpha$ to $\bar M$ is exact and let 
$\bar F:\bar M\ra H_1(M,\R)$ be a primitive of $\bar p^*\alpha$.


Then there exists a smooth map $F:M\ra\T$,
where $\T:=H_1(M,\R)/\rho(H_1(M,\Z))$ is a torus (here
$\rho:H_1(M,\Z)\ra H_1(M,\R)$ is the canonical map; denote by $\pi:H_1(M,\R)\ra
\T$ the projection), such that $\pi\circ \bar F=F\circ\bar p$. In
particular $\bar F$ is proper.  
\end{prop}
\begin{proof} The claims are straightforward, once we prove that the
objects $\bar F, \rho, F$, etc. are well-defined.  The simplest way is
to remark that, if we define the primitive $\tilde F:\tilde M\ra H_1(M,\R)$ of the pull-back to the universal cover $\tilde M$ of the chosen
"tautological" $1$-form $\alpha\in\C^\infty(\Lambda^1M\otimes
H_1(M,\R))$, then $\tilde F\circ\gamma$ is another primitive of the
same form as $\gamma^*\tilde p^*\alpha=\tilde p^*\alpha$. Recall that such a primitive $\tilde F$ can be defined by choosing $x_0\in\tilde M$ and setting
$$\tilde F(y):=\int_{x_0}^y\tilde p^*\alpha,$$
where the path from $x_0$ to $y$ in $\tilde M$ can be chosen arbitrarily, as all such paths are homotopic.

Thus $\tilde
F\circ\gamma-\tilde F=\rho(\gamma)$, for some constant in
$\rho(\gamma)\in H_1(M,\R)$,
 and we see, using the definition of $\tilde F$,
that $\rho(\gamma)$ is the integral of $\alpha$ along a loop in
$M$ representing $\gamma\in\pi_1(M)$, and thus $\rho:\pi_1(M)\to H_1(M,\R)$ is a group
homomorphism.

To see that $\rho$ induces the canonical homomorphism $H_1(M,\Z)\to H_1(M,\R)\simeq H^1(M,\R)^*$, note that, for any cohomology class $[\beta]\in H^1(M,\R)$ (de Rham cohomology) and homology class $[\gamma]\in H_1(M,\Z)$, we have (using angular brackets for the dual pairing):
$$\langle \rho(\gamma),[\beta]\rangle=\int_\gamma \langle \alpha,[\beta]\rangle=\int_\gamma [\beta]=\langle [\gamma],[\beta]\rangle.$$
Also the map $F:M\ra \T$ is induced by $\bar F$ and
the action of $\Gamma=H_1(M,\Z)$ on $\bar M$, resp.\ $H_1(M,\R)$.
\end{proof}
\begin{rem} Lemma \ref{circl} and Proposition \ref{rk1} follow from the above result because the minimal cover is a quotient of $\bar M$. The circle appearing in Lemma \ref{circl} is then a subgroup of the torus $T$ above.\end{rem}

\subsection{The Vaisman-Sasaki correspondence}
We recall (see, for example, \cite[Theorem 3]{belgun_surfaces} and \cite[Prop.~4.2]{GOOP})
that an lcK manifold is Vaisman if and only if it admits a
presentation of the form $(C(S), \Gamma)$ where $S$ is a {\em
Sasakian} manifold and $\Gamma$ is a discrete subgroup of holomorphic
homotheties of the {\em K\"ahler metric cone} over
$S$, $C(S):=\R\times S$ (see the description of the metric below) acting freely and
properly, such that the action commutes with 
the flow of $\de_t$ (here $t$ is the standard coordinate on $\R$,
extended to a smooth function on $C(S)=\R\times S$ and the function
$\lambda$ of the presentation is just $t$).  In this  case, all
presentations are of this form. 

Recall that a {\em Sasakian} metric is an odd-dimensional analog of
a K\"ahler metric, more precisely it is a contact metric manifold
$(S,g_S)$, with Reeb field $R$ (the metric dual of the contact form)
which is Killing and has unit length, and its covariant derivative defines a
$CR$ structure $J$ on $R^\perp$. We refer to \cite{DO} for
details. Here we will mainly use the fact that the {\em metric cone}
$C(S):=(\R\times S,e^t(dt^2+g_S))$ is K\"ahler (one extends $J$ by
$J\partial_t:=R$ to an (integrable) almost complex structure on
$C(S)$). We infer that, for a Vaisman metric, its pull-back to any
presentation is the product metric on $C(S)=\R\times S$, the
(conformally related) K\"ahler metric is $J$ inserted into the conic
metric, and the pull-back of the Lee form is $dt$.



This Sasakian manifold $S$  covers then every leaf $W$ of the Sasakian
foliation of $M$ (obtained by integrating the distribution $\ker\theta$), the
covering group of this covering being the normal subgroup
$\ker\rho\subset \Gamma$. Therefore, if $C(S)$ is the minimal covering
of $(M,\omega,\theta)$, then $S\simeq W$. Such an $S$ is then called
the {\em associated Sasakian manifold} to $M$.  

Conversely, if $W$ is a Sasakian manifold, the product
manifold $M = W \times S^1$ of $W$ with a circle (or, more generally,
the mapping torus of a Sasakian automorphism of $S$, see Theorem \ref{thm:prod}) admits a natural
Vaisman structure (of lcK rank $1$).

In \cite{MaMoPi}, the authors prove that if the lcK rank of a compact
Vaisman manifold $M$ is one, then $M$ is a mapping torus over its
associated Sasakian $W$ and that, hence, $W$ is compact. This result, as well as the  converse
of this statement can be retrieved as a direct consequence of Proposition \ref{rk1}, because the level sets of $\lambda$ are all diffeomorphic to $S$:
\begin{corol} \label{prop:cptSasaki} A Vaisman manifold $M$ has a
compact associated Sasakian manifold if and only if the lcK rank of $M$ is
one. 
\end{corol}

\section{Twisted Hamiltonian vector fields on lcs manifolds}\label{sec:3}

In this section we describe the infinitesimal automorphisms of an lcs
manifold, and in the next one the group actions preserving the lcs
structure. We will focus on the general lcs case; in the lcK setting, we require that {\em lcK} vector fields and transformations are {\em lcs} and {\em holomorphic}. We will briefly comment on the lcK case in Subsection \ref{lcK}.

The lcs transformations of an lcs manifold
$(M,\omega)$ are the
diffeomorphisms $\phi:M\ra M$ such that $\phi^*\omega =F\omega$, for
$F$ a function on $M$.  

\begin{prop}\label{aX} Let $X$ be a vector field on an lcs  manifold
$(M,\omega)$ whose flow  consists of lcs  transformations. Let
$(M_0,\omega_0)$ be a  presentation of   $(M,\omega)$ and lift $X$ to a vector field on $M_0$, denoted again by $X$. Then
$$\L_X\omega_0=a(X)\omega_0,$$
for some $a(X)\in\R$. \end{prop}
\begin{proof} $d\omega_0=0$ implies that $\L_X\omega = a(X)\omega_0$ has to be
closed as well, hence $a(X)$ is constant.\end{proof} 
This leads to the following definitions.

\begin{defi}[following \cite{MaMoPi}] An infinitesimal lcs
automorphism of $(M,\omega)$ is called {\em special} if its lift to a
(hence any) presentation is infinitesimal symplectomorphic.\end{defi} 

One can read off whether $X$ is special lcs without going to any
presentation, as follows.

\begin{prop}\label{spectheta} A vector field $X$ on an lcs manifold
$(M,\omega)$ is special lcs if and only if 
\begin{equation}\label{leespec}
d_\theta(\omega(X,\cdot))=\L_X\omega-\theta(X)\omega=0.
\end{equation}
\end{prop}
\begin{proof} The first equality in \eqref{leespec} is an identity,
following directy from the definition of $d_\theta$ and the Cartan
formula for $\L_X\omega$.

Consider a presentation $(M_0,\omega_0)$ of
the lcs manifold $M$, and by $X$ again the lift of $X$ to $M_0$. Being
special for $X$ is equivalent to $\L_X\omega_0=0$, where
$\omega_0=e^{-\lambda}p^*\omega$ (recall that $\lambda:M_0\ra\R$ is a
primitive of (the pull-back to $M_0$ of) the Lee form $\theta$). Then
the equality
\[
\L_X\omega_0=e^{-\lambda}(-\theta(X)\omega+\L_X\omega)
\]
implies that $X$ is special if an only if the second term in
\eqref{leespec} vanishes.
\end{proof}



\begin{defi}A vector field $X$ on an lcs manifold $(M, J,
\theta,\omega)$ is \emph{twisted Hamiltonian} if there exists a smooth
function $h_X$ on $M$ such that $\iota_X \omega = d_\theta h_X$. This
function is called the {\em twisted Hamiltonian} of the vector field
$X$.
\end{defi} In particular, a twisted Hamiltonian vector field is
lcs (even special lcs), see below. 

\begin{rem}\label{twHam1} As an example, in the lcs setting, the
{\em s}-Lee field $V$ is twisted Hamiltonian for the (constant) function
$h_V:=1$, see \eqref{aLlcs}.\end{rem}

In the definition above, we refer to $h_X$ as to {\em  the} twisted
Hamiltonian of $X$. The uniqueness of a twisted Hamiltonian for a
given lcs vector field is due to the following standard  result. 
\begin{lemma}\cite[Lemma 2.1]{MaMoPi} \label{lem:twisteddiffinj} Let
$(M,\omega,\theta)$ be a strict lcs manifold. Then the twisted
differential $d_\theta:C^\infty(M)\to \Omega^1(M)$ is injective.
\end{lemma} Clearly, if $X$ is twisted Hamiltonian, then
$\omega(X,\cdot)$ is $d_\theta$-exact, which means it is
$d_\theta$-closed, thus $X$ is special lcs by Proposition
\ref{spectheta}. Conversely, for a special lcs vector field to be
twisted Hamiltonian, the twisted cohomology class of $\omega(X,\cdot)$
(for the twisted differential $d_\theta$; this is also called the
Morse-Novikov cohomology $H^1(M,d_\theta)$) must vanish. 

Therefore, if  $H^1(M,d_\theta)=0$, then all special lcs vector fields
on $M$ are twisted Hamiltonian. However, the computation of the
Morse-Novikov cohomology is not straightforward \cite{is-ot}, so this
cohomological criterion has limited applicability. 

On the other hand, we know that special lcs vector fields exist on an
Inoue surface (which is lcK as well, and the vector field is special
lcK, i.e., holomorphic as well), that are not twisted Hamiltonian
\cite[Example 5.17]{otiman}.
\begin{rem}\label{r2.10} As shown in \cite[Rmk.~3.5]{Pilca} for the
lcK setting, it is easy to see that if $X$ is twisted Hamiltonian on
$(M,\omega)$, then it remains so for all conformally equivalent
$2$-forms $\omega'=e^f\omega$, and the corresponding twisted
Hamiltonian functions change as $h_X'=h_Xe^f$.
\end{rem} Another approach to twisted Hamiltonians of special lcs
vector fields on $(M,\omega)$ is to consider the Hamiltonians of their
lifts to the universal covering, and, possibly, to other
presentations.
\begin{prop}\label{twabel} Let $X$ be a special lcs vector field on
$(M,\omega)$. The field $X$ is twisted Hamiltonian if and only if, for
some (and then, for any) presentation $(M_0, \omega_0,\Gamma,\rho)$,
there exists a Hamiltonian $F$ for its lift to $M_0$ such that 
\begin{equation}\label{Fequiv} F\circ\gamma=e^{-\rho(\gamma)}F
\end{equation} for all $\gamma\in \Gamma$.
\end{prop}
\begin{proof} Let $F$ be an $e^{-\rho}$-equivariant Hamiltonian for the
lift of $X$, as described by \eqref{Fequiv}.  If we set $h_X:=e^{\lambda}F$, for
$\lambda:M_0\ra\R$, $d\lambda=\theta$ and
$\lambda\circ\gamma=\lambda+\rho(\gamma)$, for all $\gamma\in\Gamma$,
then $h_X$ is $\Gamma$-invariant and hence a function on $M$
satisfying $d_\theta h_X=e^\lambda dF=\omega(X,\cdot)$.

Conversely, if $X$ has $h_X$ as associated twisted Hamiltonian, then
$F:=e^{-\lambda}h_X$ is a Hamiltonian for $X$ on any presentation:
$$dF=e^{-\lambda}(-\theta+dh_X)=e^{-\lambda}d_\theta
h_X=e^{-\lambda}\omega(X,\cdot)=\omega_0(X,\cdot).$$ Moreover,
$F\circ\gamma=e^{-\rho(\gamma)}F$, as required.   
\end{proof} Of course, we know for sure that, for any special lcs vector field $X$ on $M$, its lift to the
universal covering of $M$ admits a Hamiltonian. This Hamiltonian function turns out to satisfy a more complicated $\pi_1(M)$-equivariance relation (see Proposition \ref{affine} below) than \eqref{Fequiv}, which characterizes the twisted Hamiltonian case.

In the following proposition, we denote by $\Aff_+(\R)$ the group of
orientation-preserving affine transformations of $\R$, i.e., the
transformations of the form $t \mapsto at+b$, for $a>0$ and $b \in
\R$. The group $\R_+^*$ of positive homotheties is the abelianization
of $\Aff_+(\R)$, i.e. the quotient of $\Aff_+(\R)$ by its commutator,
which is $\R$ itself, seen as the group of translations in $\R$. 

\begin{prop}\label{affine} Let $(M,\omega)$ be a strict lcs manifold
and $(\tilde M,\omega_0,\rho)$ the corresponding universal
presentation. For every special lcs vector field $X$ on $M$, there is
a group homomorphism 
$$\varphi_X:\pi_1(M)\ra \Aff_+(\R),$$
that extends $e^{-\rho}:\pi_1(M)\ra\R_+^*$.

In particular, $X$ is twisted Hamiltonian if and only if the image of
$\varphi_X$ is Abelian (thus if and only if, up to an affine
reparametrization of $\R$, $\varphi_X$ coincides with its abelianization $e^{-\rho}$).
\end{prop}
\begin{proof} As $X$ is special, there exists $F:\tilde M\ra\R$ such
that $dF=\omega_0(X,\cdot)$ (unique up to an additive constant). Thus, for every
$\gamma\in\pi_1(M)$, 
$$d(F\circ\gamma-e^{-\rho(\gamma)}F)=0,$$
as the lift of $X$ is $\gamma$-invariant and
$\gamma^*\omega_0=e^{-\rho(\gamma)}\omega_0$. If we denote by
$b_\gamma:=F\circ\gamma-e^{-\rho(\gamma)}F$ and by 
$$\varphi_X(\gamma)(t):=e^{-\rho(\gamma)}t+b_\gamma,\ t\in\R,$$
we obtain thus a map $\varphi:\pi_1(M)\ra \Aff_+(\R)$ satisfying
$$F\circ\gamma=e^{-\rho(\gamma)}F+b_\gamma=\varphi_X(\gamma)\circ F.$$
To prove that $\varphi_X$ is a group homomorphism, we compute 
$$\begin{array}{rcl}b_{\gamma\gamma'}&=&F\circ\gamma\circ\gamma'-e^{-\rho(\gamma\gamma')}F\\
&=&(F\circ\gamma-e^{-\rho(\gamma)}F)\circ\gamma'+e^{-\rho(\gamma)}(F\circ\gamma'-e^{-\rho(\gamma')}F)\\
&=&b_\gamma+e^{-\rho(\gamma)}b_{\gamma'},\end{array}$$ which is
precisely the formula for the composition of affine maps.

Finally, $X$ is twisted Hamiltonian if and only if there exists a
constant $b\in\R$ such that $F-b$ is $e^{-\rho}$-equivariant, i.e.,
$\varphi_X=T_b\circ e^{-\rho}\circ T_b^{-1}$, where $T_b:\R\ra\R$
denotes the translation with $b$.

Note that, since $\rho$ is nontrivial, this is the only possibility
for the image of $\varphi_X$ to be Abelian.
\end{proof} This proposition helps us to establish various criteria
for a special lcs vector field to be twisted Hamiltonian:  
\begin{corol}\label{minHam} A vector field on $(M,\omega)$ is twisted
Hamiltonian if and only if its lift to the minimal covering, or to any
presentation with Abelian group $\Gamma$, is Hamiltonian.\end{corol}
\begin{proof}If $X$ has $h_X:M\ra\R$ as twisted Hamiltonian, then its
lift to any  presentation $(M_0,\omega_0,\Gamma)$ admits
$h_Xe^{-\lambda}$ as a Hamiltonian (as usual, $d\lambda=\theta$).

Conversely, suppose $X$ is Hamiltonian on $(M_0,\omega_0,\Gamma)$,
with $\Gamma$ Abelian. Then, following the proof of Proposition
\ref{affine}, we obtain a group homomorphism from $\Gamma$ to
$\Aff_+(\R)$. As its image is Abelian, it follows that $X$ has an
$e^{-\rho}$-equivariant Hamiltonian, thus $X$ is twisted Hamiltonian
by Proposition \ref{twabel}.
\end{proof} We give some examples of  properties of the fundamental
group of $M$ that imply that every special lcs vector field is twisted
Hamiltonian:
\begin{corol}\label{ss,nilp} If the fundamental group of $M$ is
semisimple or nilpotent, then every special lcs vector field on
$(M,\omega)$ is twisted Hamiltonian.
\end{corol}
\begin{proof} With the above hypotheses, the image of $\varphi_X$ in
$\Aff_+(\R)$ is also semisimple or nilpotent. This, in turn, we can
show it implies that this image $H\subset \Aff_+(\R)$ is Abelian: let
$\varphi_i(t):=a_it+b_i$, $i=1,2$ be two elements in $\Aff_+(\R)$. We
compute their commutator:
\begin{align*}
\varphi_1&\circ\varphi_2\circ\varphi_1^{-1}\circ\varphi_2^{-1}(t)\\ &=
a_1(a_2a_1^{-1}(a_2^{-1}(t-b_2)-b_1)+b_2)+b_1=t+b_1(1-a_2)-b_2(1-a_1).
\end{align*} This implies that
\begin{itemize}
\item The commutator of any subgroup $H$ of $\Aff_+(\R)$ is included
in the subgroup of translations and is thus Abelian;
\item If $\varphi_1$, with $a_1=1$ and $b_1\ne 0$, is in the center of
$H$, then $H$ consists only of translations;
\item If $\varphi_1$, with $a_1\ne 1$, is in the center of $H$, then
all elements of $H$ must fix the unique fixed point of $\varphi_1$, in
particular $H$ contains no nontrivial translations and thus its
commutator is trivial.
\end{itemize} If $H$ is nilpotent, then its center is nontrivial, thus
$H$ is Abelian. If $H$ is a direct product of simple groups, then its
commutator is the product of its nonabelian factors (if any), but this
has to be contained in the subgroup of translations, it (and therefore
$H$ itself) is thus Abelian. 
\end{proof} Another application of Proposition \ref{affine} is when
the commutator of $\pi_1(M)$ is ``small enough'':
\begin{corol}\label{comm-small} Suppose the torsion-free part of the
abelianization of the commutator of $\pi_1(M)$ is trivial or
cyclic. Then every special lcs vector field on $(M,\omega)$ is twisted
Hamiltonian.\end{corol}
\begin{proof} If we denote by $H$ the image of $\varphi_X$, for $X$ a
special lcs vector field, then the restriction of $\varphi_X$ to the
commutator of $\pi_1(M)$ is a group homomorphism into $\R$, thus it
induces a map $\phi$ from the abelianization of the commutator of
$\pi_1(M)$ to $\R$. Its image is the commutator $H_1$ of $H$, and it
is a subgroup of $\R$ generated by at most $1$ element.

If the abelianization of the commutator of $\pi_1(M)$ is pure torsion,
$H_1$ is trivial thus $H$ is Abelian.

If we suppose that $H_1$ is cyclic, generated by $b\in\R$, then we use
that $H_1$ must be invariant by conjugation with non-translation
elements of $H$. But conjugation of the translation $t\mapsto t+b$ by
$t\mapsto at+b$, with $a<1$ (such elements exist in $H$), is the
translation by $ab$, which does not belong to the group generated by
$b$. Thus, the case, $H_1$ cyclic does not occur, hence $H$ is
Abelian.\end{proof} 
There is another instance where a special lcs vector field is twisted Hamiltonian that will turn out useful in the Vaisman case:
\begin{prop}\label{dexact} Let $(M,\omega)$ be lcs and let $\omega=d_\theta\eta$, where $\theta$ is the Lee form of $\omega$. Then every vector field $X$ on $M$ that satisfies $\theta(X)\equiv 0$ and  whose flow preserves $\beta$ admits $-\beta(X)$ as the twisted Hamiltonian.\end{prop}
\begin{proof} We compute
  $$d_\theta(-\beta(X))=-d(\beta(X))+\beta(X)\theta=d\beta(X,\cdot)-\theta\wedge\beta(X,\cdot)=\omega(X,\cdot),$$
  where we have used the Cartan formula for $\L_X\beta=0$ and that $\theta(X)=0$. $X$ is thus twisted Hamiltonian.\end{proof}
\begin{corol}\label{twHamVais} A special lcK (i.e., special lcs and holomorphic) vector field on a Vaisman manifold is twisted Hamiltonian.\end{corol}
\begin{proof}It is well-known (see, for example \cite{DO}), that the Vaisman $2$-form $\omega$ equals $d_\theta(\theta\circ J)$. if $X$ is special lcs, then it is conformal Killing, thus by \cite[Proposition
  3.8]{MaMoPi} its flow preserves the Gauduchon metric, which is the Vaisman metric itself. Thus $\L_X\omega=0$, which implies $\L_X\theta=0$ and $\L_X(\theta\circ J)=0$, because $X$ is holomorphic. As $\theta(X)$ is constant and $X$ is special lcK (lcs), then $\theta(X)\equiv 0$ and we conclude by Proposition \ref{dexact}.\end{proof}

\section{Group actions on lcs manifolds}\label{sec:4}
Let $G$ be a Lie group with Lie algebra $\lie g$ acting on an lcs
manifold $(M,\omega)$. We will denote by $\bar X\in  \mathcal{X}(M)$
the fumdamental vector field associated to $X\in\lie g$.

\begin{defi} The action of $G$ is {\em lcs} on $(M,\omega)$ if and only
  if $g^*\omega$ is proportional with $\omega$, for $g\in G$. The
  action is {\em special lcs}, resp.\ \emph{(weakly) twisted Hamiltonian}
if every $X \in \lie g$ induces a special lcs, resp.\ twisted
Hamiltonian field $\bar X$ on $M$. The action is of \emph{Lee type} if the {\em
  s}-Lee field is induced by an element of the Lie algebra of $G$.
\end{defi} 
The adverb {\em weakly} refers to the possible lack of an
equivariant {\em twisted moment map} (see below). We will however omit
it in the sequel, and refer simply to {\em twisted Hamiltonian}
actions. In \cite{Pilca}, the author requires for the definition of
{\em twisted Hamiltonian} an additional condition with respect to the
{\em weakly twisted Hamiltonian} setting. This condition is that the
moment map, seen as a map from $\lie g$ to $C^\infty(M)$, is a Lie
algebra homomorphism for the {\em twisted Poisson structure}
\begin{equation}\label{B} \{f,g\}_\theta:=\omega^{-1}(d_\theta
f,d_\theta g).
\end{equation} If fact, it is easy to see that this condition is
automatically satisfied on strict lcs manifolds, which we consider here,
hence we omit it in the definition. For a complete study of equivalent
formulations of \eqref{B}, and also for their automatic occurence on a
strict lcs manifold, see \cite[Section 3]{HallerRybicki}. 

\begin{prop}\label{closure-sLee} Suppose that the closure of the flow
  of the {\em s}-Lee field $V$ on a compact lcs manifold $(M,\omega)$ is a compact
  torus $T$ (equivalently, the flow of $V$ consists of isometries of
  some Riemannian metric on $M$). Then the action of $T$ is special lcs.\end{prop}
\begin{proof}
  The flow of a vector field $V$ on a
  compact manifold $M$ has a compact closure $T$ in the diffeomorphism group if and only if $V$ is Killing with respect to some Riemannian metric. Clearly, as the closure of an abelian group, $T$ is Abelian, thus a torus. We will compare $T$, the closure of the flow of $V$ on $M$, with $\hat T$, the closure of the flow of the lift $\hat V$ of $V$ on the minimal covering $\hat M$ (in the isometry group of the lift of some $T$-invariant Riemannian metric). 
  \smallskip
  
  Two possibilities occur: either $\hat T$ is a torus as well (i.e., the closure of the flow of $\hat V$ is compact as well), or $\hat T$ is a closed, non-compact subgroup of the isometry group of $\hat M$ isomorphic to $\R$. We will first show that the second case  does not occur.

It is known that a closed subgroup of the isometry group of a Riemannian manifold acts properly, see \cite[Proposition 5.2.4]{PalaisTerng}, and hence all isotropy groups of the $\hat T$-action on $\hat M$ are compact. To show compactness of $\hat T$ it is therefore sufficient to show compactness of the $\hat T$-orbits (in fact, of a single $\hat T$-orbit).


 As  $\omega$ is invariant under the flow of the {\em s}-Lee field $V$, it is invariant under its closure $T$ as well. We choose a $T$-invariant associated almost Hermitian metric $g:=\omega(\cdot,J\cdot)$ (for $J$ the corresponding almost complex structure), and consider the (vector-valued) $1$-form $\alpha$ in Proposition \ref{Fprop} to be harmonic with respect to $g$. Its components are thus harmonic forms and therefore $\alpha(V)$ is constant, because $V$ is Killing (see, e.g. \cite[Lemma 5.1]{MaMoPi}). If we show that $\alpha(V)=0$, then the orbits of $\bar V$ (the lift of $V$) on the homological covering $\bar M$ are contained in the fibers of the proper map $\bar F:\bar M\ra H_1(M,\R)$, see Proposition \ref{Fprop}, and are thus relatively compact.  Of course, this means that the orbits of $\hat V$ on $\hat M$, which is a quotient of $\bar M$, are relatively compact.
  
  So we need to prove that, for every harmonic form $\beta$, the constant $\beta(V)$ is zero. Suppose $\beta(V)>0$, then $\int_M\beta(V)>0$, and this is the integral of the scalar product of $\beta$ with $V^\flat$, the metric dual of $V$, which is $-\theta\circ J$, where $J$ is the almost complex structure compatible with $\omega$ and the metric.

  Now, on an almost Hermitian manifold of dimension $2n$, the {\em Hodge star operator} $*:\Lambda^kM\ra\Lambda^{2n-k}M$ has the following expression in degree 1: for all $X\in TM$ we have
\begin{equation}\label{star}*X^\flat=-\frac{1}{(n-1)!}\omega^{n-1} \wedge\omega(X,\cdot).\end{equation}
For $X=V$, we obtain:
$$*V^\flat=-\frac1{(n-1)!}\omega^{n-1}\wedge\theta=\frac1{(n-1)!(n-1)}d(\omega^{n-1}),$$
which is thus exact, i.e., $V^\flat$ is co-exact. But the $L^2$-scalar product of a co-exact form with a harmonic form vanishes, thus $\alpha(V)=0$ for all harmonic 1-forms $\alpha$ on $M$, thus the orbits of $V$ on $\bar M$ are included in the fibers of $\bar F$, which are compact.

This shows that the closure of the flow of $\hat V$ on $\hat M$ is compact, as claimed.
\smallskip

  From \eqref{leespec}, the action of $T$ is special lcs if and only if $\L_{\bar X}\omega=\theta(\bar X)\omega$ for all $X\in\lie t$. But $\L_{\bar X}\omega=0$ implies $\L_{\bar X}\theta=0$, thus $\theta(\bar X)=d\lambda(\bar X)$ is constant for all $X\in\lie t$, where $\lambda:\hat M\ra \R$ is a primitive of $\theta$.

  As $\hat V$ is compact, for every $X\in \lie t$ the constant $\theta(\bar X)=d\lambda(\bar X)$ vanishes, as $\lambda$ is bounded on the orbits of $\hat T$ (which were seen to be compact above). Thus the $T$-action is special lcs.
\end{proof}
\begin{corol} A compact lcs manifold $M$ admits a special lcs torus action of Lee type if and only if the {\em s}-Lee vector field $V$ is Killing with respect to some Riemannian metric.\end{corol}
\begin{corol} The torus obtained as the closure of the flow of the {\em s}-Lee vector field $V$ on a compact Vaisman manifold $M$ acts on $M$ in a holomorphic, twisted Hamiltonian fashion.\end{corol}
\begin{proof} As $V$ is a Killing vector field, the torus $T$ from Proposition \ref{closure-sLee} consists of isometries that preserve the Vaisman $2$-form $\omega$.
  All fundamental vector fields $\bar X$, for $X\in \lie t$ are thus holomorphic and (from Proposition \ref{closure-sLee}) special lcs. By Corollary \ref{twHamVais} they are twisted Hamiltonian.\end{proof}
 \begin{rem} Note that the projection from $\hat M$ to $M$ induces a group homomorphism  $\psi:\hat T\ra T$, whose image contains the flow of $V$, thus it is dense in $T$. We have shown that $\hat T$ is compact, thus $\psi(\hat T)$ is compact and dense, hence $\psi:\hat T\ra T$ is surjective.

  On the other hand, all elements in $\ker\psi$ (which are of finite order) must be induced by the deck transformation group $\hat \Gamma=\pi_1(X)/\Gamma$, which is free abelian. This means that $\psi:\hat T\ra T$ is a Lie group isomorphism.\end{rem}
  We show now that it is a feature of special lcs actions of $G$ to lift to actions of $G$ on the minimal covering (thus $G_0=G$ in Remark \ref{lift-gen}).
\begin{lemma}\label{liftG} Consider a special lcs action of a
connected Lie group $G$ on the lcs manifold $(M,\omega)$. Then this
action lifts to a symplectic action of $G$ on $(\hat
M,\omega_0)$.\end{lemma}
\begin{proof} 

  By lifting the (special lcs) fundamental fields of $G$ to the minimal covering $\hat M$, we obtain a symplectic action of $\tilde G$, the universal covering of $G$ on $\hat M$.  Denote by $p:\tilde G\ra G$ the projection. 
  If $p(\tilde g)=1\in G$ for some $\tilde g\in\tilde G$, then $\tilde g$ maps the fibers of $\pi$ to themselves.
  
  But the covering $\pi:\hat M\ra M$ is Galois, with group
$\Gamma$, thus for all $x\in \hat M$ there exists a unique
$\gamma_x\in\Gamma$ such that $\tilde g.x=\gamma_x(x)$. By continuity,
$\gamma_x$ is actually independent of $x$ and thus equal to a single
element $\gamma\in\Gamma$.

All nontrivial elements of $\Gamma$ are
{\em strict} symplectic homotheties, but $x\mapsto \tilde g.x$ is a
symplectomorphism, thus $\gamma=1$ and $\tilde g$ acts trivially on $\hat M$. Thus the action of $\tilde G$ factors to an action of $G$ on $\hat M$ as claimed.\end{proof} 
The above
result has been proven previously for the case of a compact torus in
\cite{istrati_toric} and, in the lcK setting, for compact $G$, in
\cite{MaMoPi}.
\smallskip 

We now recall the map that bundles together all the twisted Hamiltonian functions, in analogy with symplectic geometry:

\begin{defi} Let $G$ be a Lie group that acts on a lcs manifold and
such that every $X \in \lie g$ is twisted Hamiltonian. A map $\mu
\colon M \to \lie g^*$ such that $d_\theta \mu^X = \iota_{\bar X}
\omega$ is a \emph{twisted moment map}.
\end{defi} 

In the symplectic setting, it is known that the moment map of a
Hamiltonian $G$-action on a symplectic manifold is unique up to
constants in $[ \lie g, \lie g]$. In the setting of twisted Hamiltonian actions on strict lcs manifolds, the uniqueness of the twisted moment map follows automatically from Lemma \ref{lem:twisteddiffinj}.
\begin{rem}\label{lift-gen} If a connected Lie group $G$ acts by special lcs
transformations on $(M,\omega)$, then its fundamental fields $\bar X$, for
$X\in\lie g$, lift (we denote their lifts to $M_0$ equally by $\bar
X$) to induce on any presentation $(M_0,\omega_0)$ a symplectic
action of a covering $G_0$ of the group $G$. Proposition \ref{twabel}
implies that the action of $G$ is twisted Hamiltonian, if and only if
the action of $G_0$ is Hamiltonian and admits a (necessarily unique)
$e^{-\rho}$-equivariant  (w.r.t. the action of $\Gamma$ on $M_0$)
moment map $\mu_0:M_0\ra\lie g$. This moment map (that we call {\em the
symplectic moment map} of a twisted Hamiltonian action) is related,
cf. \eqref{Fequiv}, to the twisted moment map $\mu$ of the action of
$G$ on $M$: 
  \begin{equation}\label{mu0mu}\mu_0=e^{-\lambda}\mu\circ
p,\end{equation} where $p:M_0\ra M$ denotes the covering map. We see
thus that, while the twisted moment map depends on the choice of
$\omega$ in its conformal class (and of the Lee form $\theta$), the
symplectic moment map depends only on the lcs structure (i.e., the
conformal class of $\omega$ and, therefore, on the cohomology class of
the Lee form $\theta$).\end{rem}

As in symplectic geometry, we consider the problem of $G$-equivariance
of the moment map. First, we show the $G_0$-equivariance of the
symplectic moment map. 
\begin{prop}\label{eqiv}
Let a connected Lie group $G$ act in twisted Hamiltonian fashion on
$(M,\omega)$.  For every presentation $(M_0,\omega_0)$, if we denote
by $\mu_0:M_0\ra\lie g$ the moment map of the symplectic action of
$G_0$ on $M_0$, then $\mu_0$ is $G_0$-equivariant w.r.t.\ the action on
$M_0$ an the coadjoint action on $\lie g$.\end{prop}
\begin{proof} As $G$ and its covering $G_0$ are connected, the
  equivariance of $\mu_0$ holds if and only if
\begin{equation}\label{eqv}\bar X.(\mu_0)_Y=(\mu_0)_{[X,Y]}\end{equation}
  for all $X,Y\in \lie g$.
To prove this, we show first that the left hand side is a Hamiltonian
for the fundamental field defined by $[X,Y]$ on $M_0$, in other words,
that the difference to the right hand side is a constant:
$$d(\bar X.(\mu_0)_Y)=d(\omega_0(\bar Y,\bar X))=\L_{\bar
  X}(\omega_0(\bar Y,\cdot))=\omega_0([\bar X,\bar
Y],\cdot)=d(\mu_0)_{[X,Y]},$$
because $\omega_0$ is closed and $\bar X$-invariant. Moreover, $\bar
X.(\mu_0)_Y$ is again $e^{-\rho}$-equivariant, thus it coincides with
$(\mu_0)_{[X,Y]}$.  Thus $X.\mu_0=\mbox{ad}_X^*(\mu_0)$.\end{proof}
Let us now discuss the $G$-equivariance of the twisted
moment map.
\begin{prop} \label{equivmmap} Let $(M, \omega, \theta)$ be a strict
lcs manifold  and $G$ be a connected Lie group acting in twisted
Hamiltonian fashion on $M$. Then the following conditions are
equivalent:
\begin{enumerate}
\item The twisted moment map $\mu: M\ra\lie g^*$ is $G$-equivariant,
i.e.,
$$\mu(g\cdot x)= \mathrm{Ad}_{g^{-1}}^*(\mu(x))$$
for all $g\in G$.
\item $\omega$ is $G$-invariant. 
\item $\theta(\bar X)=0$ for all $X\in \lie g$.
\end{enumerate}
\end{prop}
\begin{proof} The equivalence of the first two conditions follows from
  the equivariance of $\mu_0$ (Proposition \ref{eqiv}) and equation
  \eqref{mu0mu}, see also \cite[Proposition 2]{HallerRybicki} for
  another proof. 

The equivalence of the second and
third condition is clear because by Proposition \ref{spectheta} we
have $\L_{\bar X}\omega = \theta(\bar X)\omega$ for the fundamental
vector fields $X$ of the action.
\end{proof} If the twisted moment map is $G$-equivariant, then we call
the $G$-action {\em strongly twisted Hamiltonian}.

\begin{rem} In \cite{HallerRybicki}, the authors use a different,
weaker definition of $G$-equivariance:
  $$\mu(g\cdot x)=a_g \mathrm{Ad}_{g^{-1}}^*(\mu(x))$$
for all $g\in G$, where $a_g\in (0,\infty)$ also occurs as a factor in
the conformal $G$-invariance of $\omega$:
$g^*\omega=a_g^{-1}\omega$. We would rather call this notion {\em
conformal $G$-equivariance}, and, in fact, it is shown already in \cite{HallerRybicki} to be
equivalent to \eqref{B} and thus to hold automatically on every strict
lcs manifold, \cite[Proposition 2]{HallerRybicki}. 
\end{rem}  
For a given twisted Hamiltonian action of a compact
group $G$ on $(M,\omega)$, one can integrate over $G$ (for a
bi-invariant volume form on $G$) the function
$g\mapsto g^*\omega$ and get a conformally equivalent $2$-form
$\omega'$ which is $G$-invariant. (Recall that $g^*\omega$ is conformal to $\omega$ by definition of an lcs transformation.) {\em A priori}, one may  lose the Lee type property in the process. The following shows that this is not the case if $G$ is a torus:

\begin{prop}\label{Tequiv} Let $T$ be a torus acting on a strict lcs manifold $(M,\omega)$ by special lcs transformations and suppose the action is of Lee type. Then $\omega$ is $T$-invariant.\end{prop}
\begin{proof} By Proposition \ref{equivmmap} it suffices to show that $\theta(\bar X)=0$ for all $X\in \lie t$. We have $[\bar X,V]=0$ for all $X\in \lie t$, where $V$ is the {\em s}-Lee vector field. We write also
\begin{equation}\label{dthetaX}  d(\theta(\bar X))=\L_{\bar X}\theta=-\L_{\bar X}(\omega(V,\cdot))=-(\L_{\bar X}\omega)(V,\cdot)=\theta(\bar X)\theta.
\end{equation}
This implies $\theta(\bar X)=0$ by the following lemma (see also the argument
from the proof of \cite[Lemma 2.1]{MaMoPi})
\end{proof}
\begin{lemma}\label{lem2.1}
Let $N$ be a connected manifold, $\eta$ a closed, non-exact form on
$N$, and $f:N\ra\R$ a function satisfying $df=f\eta$ on $N$. Then
$f\equiv 0$.\end{lemma}
\begin{proof}Let $\beta:\tilde N\ra\R$ such that $d\beta=p^*\eta$, for
  $p:\tilde N\ra N$ the universal covering on $N$. Then we have
$$d(e^{-\beta}f\circ p)=e^{-\beta}(-d\beta f\circ p+df\circ p)=0$$
on $\tilde N$, thus $e^{-\beta}f\circ p=k$ is constant on $\tilde N$. This
implies $f\circ p=ke^\beta$, but $\beta$ and $e^\beta$ are not
$\pi_1(N)$-invariant (otherwise $\eta$ would be exact), so the only
possibility is $k=0$.\end{proof}

\begin{corol} A twisted Hamiltonian action of Lee type of a torus $T$ has a $T$-equivariant twisted moment map.\end{corol} 


As we are interested in actions of compact groups,
mainly in torus actions, we may consider actions that are are extended
by continuity from
actions of some dense subgroups (like a torus action can be the closure of the
flow of one vector field). We show the following:
\begin{prop}\label{compactify} Let $H\subset G$ be a dense subgroup in a compact Lie group $G$ and assume that $H$ acts by special lcs
  transformations on a compact lcs manifold $(M,\omega)$ of lcs rank $1$. Then $G$ acts by special lcs transformations of
  $(M,\omega)$. \end{prop}
\begin{proof} It is clear that the action can be extended to $G$ and
  that this action is lcs (here we need that $M$ is compact). 
  Possibly after integrating $g^*\omega$ over 
  $G$ as described above, we can suppose that $\omega$ is
  $G$-invariant, thus $\L_{\bar X}\theta=0$, i.e., $\theta(\bar X)$ is 
  constant for any $X\in\lie
  g$. Moreover, if $X\in \lie h$, then $\theta(\bar X)=0$ from
  Proposition \ref{spectheta}. 

But this is equivalent, cf.\ Lemma \ref{circl}, to $d\bar\lambda(\bar X)=0$, where we recall that $\bar\lambda:M\ra S^1$ is induced by a
  primitive $\lambda:\hat M\ra \R$ of $\theta$, in the case of lcs
  rank $1$. This means that the function $\bar\lambda$ is constant along the $H$-orbits; by continuity it is constant along the $G$-orbits. Hence $\theta(\bar X) = 0$ for all $X\in \lie g$ which shows that the $G$-action is special lcs by Proposition \ref{spectheta}. 
\end{proof}
\begin{rem} We could have tried to apply the strategy from Proposition \ref{closure-sLee} to show that the action of $G$ lifts to $\hat M$, and possibly drop the restriction on the lcs rank.

  For this, one would first lift the action of $H$ to $\hat M$ (this is possible, see Lemma \ref{liftG}), then one would need to show that its orbits are included in the fibers of the proper map $\bar F:\bar M\ra H_1(M,\R)$, thus are relatively compact. This was achieved in Proposition \ref{closure-sLee} because $*V^\flat$ is exact (for $V$ the {\em s}-Lee vector field), but if we consider in \eqref{star} an arbitrary vector field $\bar X$, for $X\in\lie g$, then $*X^\flat$ is not necessarily exact, therefore the map $\bar F$ might not necessarily be constant along the $H$-orbits and the method fails.
\end{rem}
\begin{rem}\label{close-twHam} The property of being twisted Hamiltonian does not
  necessarily extend from a dense (disconnected) subgroup $H$ to $G$: if, for example $H\simeq \Z$ is dense in a circle $S^1$ acting on an Inoue surface
  by special lcK transformations (see \cite[Example 5.17]{otiman}),
  then the action of the discrete group $H$ is trivially twisted
  Hamiltonian, while the action of $S^1$ is not.

However, if we consider the special case where $H$ is a group of lcK transformations on a compact Vaisman manifold of rank $1$, then $G$ is also special lcK and thus, by Corollary \ref{twHamVais}, twisted Hamiltonian. \end{rem}
\subsection{Remarks concerning the lcK setting}\label{lcK} An lcK transformation
is conformal and biholomorphic, or, equivalently, lcs and
biholomorphic. In particular, all the above types of vector fields
(lcs, special lcs and twisted Hamiltonian) can be encountered in lcK
geometry, but they are lcK if and only if they are, additionally,
holomorphic (as usual, we say that a real vector field $X$ is holomorphic
if $\L_XJ=0$).

Moreover, a group $G$ acts on $(M,\omega,J)$ by lcK transformations
(possibly special lcK or twisted Hamiltonian lcK) if, additionally to the
conditions considered above in the lcs setting, $G$ acts by biholomorphisms of
$(M,J)$. This is a serious restriction for actions of Lee type,
because the Lee field needs to be holomorphic. This holds for Vaisman
structures, but there are very few non-Vaisman lcK examples with
holomorphic Lee field, see \cite{Mor2Ornea}. 

On the other hand, the lcK action of a group $G$ on a compact lcK
manifold $(M,\omega,J)$ can always be compactified:
$G$, acting by conformal biholomorphisms on $(M,J,g)$, acts by
isometries with respect to the Gauduchon metric, thus the closure $\bar G$ of $G$ in the (compact) group of isometries of this metric is compact. The Lee type
condition (although rare, see \cite{Mor2Ornea}) clearly passes to $\bar
G$, but it is not automatic that the action of $\bar G$ is still
special or twisted Hamiltonian (if the action of $G$  was of this kind), see, however, Corollary \ref{twHamVais} and Remark \ref{close-twHam} above.




\section{Convexity of the twisted moment map}\label{sec:5}

The goal of this section is to prove an lcs analog of the classical
symplectic convexity theorem of Atiyah~\cite{atiyah_conv} and
Guillemin-Sternberg~\cite{techniques}:

\begin{thm} \label{thm:convmu} Consider a twisted Hamiltonian action
of Lee type of a compact torus $T$ on a compact lcs manifold $(M,
\omega, \theta)$, with twisted moment map $\mu \colon M \rightarrow
\lie t^*$. Then $\mu:M\rightarrow \mu(M)$ is an open map with
connected level sets, and convex image.
\end{thm}

In order to prove this theorem we will use the
\emph{local-global-principle} \cite{HilgertNeebPlank}, see also
\cite{BirteaOrtegaRatiu, BjoKar}. In a form suitable for our purposes
it reads
\begin{thm}[Local-global principle] \label{thm:locglob} Let $X$ be a
connected Hausdorff space and $\mu \colon X \rightarrow \R^k$ a proper
map. If every $x \in X$ admits a neighborhood $U$ such that $\mu|_U :
U \rightarrow \mu(U)$ is open and convex, then $\mu \colon X
\rightarrow \mu(X)$ is open and convex.
\end{thm}

For completeness, we recall the notion of convexity for maps which is
used here.

\begin{defi} A continuous path $c : [0, 1] \rightarrow \R^k$ is
\emph{monotone straight} if it is a weakly monotone parameterization
of a (possibly degenerate) closed interval, that is if $c(t_2)$
belongs to the segment from $c(t_1)$ to $c(t_3)$ whenever $0 \leq t_1
\leq t_2 \leq t_3 \leq 1$.

A continuous map $f\colon X \rightarrow \R^k$ is \emph{convex} if for
all $x, y \in X$ there exists a continuous path $c\colon [0,1]
\rightarrow X$ such that $c(0)=x, c(1)=y$ and $f \circ c$ is monotone
straight.
\end{defi}

Note that this is different from the notion of convexity of real
valued functions. We will also call a map $\mu$ satisfying the
assumptions of Theorem \ref{thm:locglob} \emph{locally open} and
\emph{locally convex}. 

This purely topological principle was used to reduce the statement of
the classical symplectic convexity to the following well-known local
statement, which one obtains by establishing a local normal form for
the moment map.
\begin{lemma} \label{lemmastep1}[see e.g.~Prop.~29 of~\cite{BjoKar}] A
moment map of a Hamiltonian action of a torus $T$ on a not necessarily
compact symplectic manifold $M$ is locally open and locally convex.
\end{lemma}

In the situation of Theorem \ref{thm:convmu} we will achieve convexity
of the twisted moment map through the following steps:
\begin{enumerate}[1.]
\item \label{step1proof} Lift the $T$-action to a Hamiltonian action
on the minimal covering using Proposition \ref{liftG}. Its moment
map $\hat \mu$ is locally open and locally convex by Lemma
\ref{lemmastep1}.
\item \label{step2proof} Show that $\mu$ is locally open and locally
convex using the projection $\pi \colon \hat M \rightarrow M$ and the
relation \eqref{mu0mu} between the twisted vs. symplectic moment map.
\item \label{step3proof}  Conclude that $\mu$ is open and convex using
the local-global principle in~Theorem~\ref{thm:locglob}.
\end{enumerate} 

Observe that in order to apply the local-global principle, it is
essential that the map in question is proper. This is automatic for
$\mu$, as we assume the manifold $M$ to be compact, but not for
$\hat\mu$ (see below).


To finish the proof of Theorem \ref{thm:convmu} it remains to show the
second step above. For this, we now use the assumption on the action
that it is of Lee type: let $Y\in \lie t$ be the element that induces
the {\em s}-Lee vector field. By Remark \ref{twHam1}, $\mu:M\to \lie t^*$
takes values in the affine subspace $\{\alpha\in \lie t^*\mid
\alpha(Y)=1\}$. Recall that by \eqref{mu0mu}, the moment
maps $\mu$ and $\hat \mu$ are related by
\[ \hat \mu = e^{-\lambda} \mu\circ \pi,
\] where $\pi\colon \hat M \rightarrow M$ is the minimal
covering. This shows that 
\begin{equation}
\label{eq:relPsi} \mu \circ \pi = \Psi \circ \hat \mu,
\end{equation} where the rescaling map $\Psi \colon \{ \alpha \mid
\alpha(Y) > 0 \}=:\call H \rightarrow \lie t^*$ with image
$\{\alpha\in \lie t^*\mid \alpha(Y)=1\}$ is given by
\[ \Psi(\alpha) = \frac{\alpha}{\alpha(Y)}.
\]

We have the following fact.
\begin{lemma} \label{lemmaPsi} The map $\Psi$ is open and maps weakly
monotone parameterizations of segments in $\call H$ to weakly monotone
parameterizations of segments in $\lie t^*$.
\end{lemma}

\begin{proof} Openness follows by taking a basis $\{e_i \}$ of $\lie
t$ and its dual $\{e^i\}$ such that $Y = e_1$. In these coordinates we
have 
\[ \Psi(y_1, \ldots, y_k) = \biggl (1, \frac{y_2}{y_1}, \ldots,
\frac{y_k}{y_1} \biggr)
\] and we can see that $\Psi$ is open.

Clearly, $\Psi$ is a projective transformation (it is the central
projection from the origin onto an affine hyperplane), and as such it
maps lines into lines. By continuity, it maps interiors of segments
into interiors of segments, thus it is monotone straight.  
$\gamma\colon [0,1] \rightarrow \call H$ be a monotone
parameterization of a segment. This means that for all $0 \leq t_1
\leq t_2 \leq t_3 \leq 1$, the point $\gamma(t_2)$ belongs to the
segment $[\gamma(t_1), \gamma(t_3)]$. 
\end{proof}

Using the local openness and convexity of $\hat \mu$, we can deduce
now the same properties for $\mu$.

\begin{lemma}\label{lemma:loccvxmu} The map $\mu\colon M \rightarrow
\lie t^*$ is locally open and locally convex.
\end{lemma}

\begin{proof} Let $p \in M$ and choose a point $\hat p \in \hat M$ in
the fiber over $p$. By Lemma \ref{lemmastep1}, there exists a
neighborhood $\hat U \subset \hat M$ of $\hat p$ such that the
restriction of $\hat \mu$ to $\hat U$ is locally open and locally
convex.

Let $U := \pi( \hat U)$, let $x, y \in U$ and choose $\hat x, \hat y
\in \hat U$ in their fibers. The convexity of $\left.\hat
\mu\right|_{\hat U}$ means that there exists a continuous path $\hat
\gamma \colon [0,1] \rightarrow \hat U$ such that $\hat \gamma(0) =
\hat x$, $\hat \gamma(1) = \hat y$ and $\hat \mu \circ \hat \gamma$ is
a monotone parameterization. Let $\gamma = \hat \gamma \circ \pi$ be
the projection of the path $\hat \gamma$, connecting $x$ to $y$.

From the relation \eqref{eq:relPsi}, we have that $\mu \circ \gamma =
\Psi \circ \hat \mu \circ \hat \gamma$ and by Lemmas \ref{lemmaPsi}
and \ref{lemmastep1} we conclude that  $\mu \circ \gamma$ is a
monotone parameterization of a segment. The local openness of $\mu$
also follows from the one of $\hat \mu$.
\end{proof}

As explained above, this concludes the proof of
Theorem~\ref{thm:convmu}.

Conversely, the Lee type property can be read off in terms of the image of the twisted moment map.
\begin{prop}\label{convex-tw-affine}
  Let $M$ be an lcs manifold and $T$ a torus action on it (effectively) in a twisted Hamiltonian fashion and having $\mu:M\ra\lie t^*$ as the twisted moment map. Suppose $\mu(M)\subset \lie t^*$ is convex and has no interior points. Then the action of $T$ on $M$ is of Lee type.\end{prop}
\begin{proof} If $\mu(M)$ is convex, then there exists an affine subspace $W\subset \lie t^*$ such that $\mu(M)\subset W$, and $W$ is minimal with this property. By assumption, $W\ne \lie t^*$.

  If $W\subset \lie t^*$ has codimension at least $2$, or if $0\in W$, there exists $X\in \lie t$, $X\ne 0$, such that $\mu(M)$ lies in the annihilator of $X$, or, equivalently, $\mu_X\equiv 0$. This is, however, impossible since the action is effective. Therefore, $W$ is an affine hyperplane, thus there exists a unique element $V\in \lie t$ such that $W=\{\alpha\in\lie t^*\ |\ \alpha(V)=1\}$. But then $\mu_V\equiv 1$, which means that $V$ is the {\em s}-Lee vector field.
\end{proof}
Note that convexity of $\mu$ alone does not characterize the Lee type property: as an example, consider $\lie t_0$ a complement in $\lie t$ to $\R V$, where $T$ is a torus acting in twisted Hamiltonian fashion of Lee type on $M$, with twisted moment map $\mu$. For a countable choice of $\lie t_0$, it is the Lie algebra of a subtorus $T_0\subset T$. The twisted moment map $\mu_0$ of the action of $T_0$ is obtained from $\mu$ after projecting onto $\lie t^*_0=\lie t^*/\mbox{Ann}(\lie t_0)$. It is thus convex as well. On the other hand, the corresponding projection induces an isomorphism from $\mathcal{H}$ to $\lie t^*_0$, thus the image of $\mu_0$ contains a nonempty open set. Depending on the choice of $T_0$, $0$ may be in the image of $\mu_0$ or not.\smallskip

Finally, we wish to show that the moment
image is a convex polytope:

\begin{prop} \label{prop:imagepolytope} In the situation of Theorem
\ref{thm:convmu}, the moment image $\mu(M)$ is a convex polytope. It
is the convex hull of $\mu(C)$, where 
  \[ C = \{p\in M\mid \bar X_p \in \R\cdot V_p \text{ for all } X\in
\lie t\}  
  \] is the set of points whose $T$-orbits equal their $V$-orbits,
where $V$ is the {\em s}-Lee vector field. (Thus the $T$-orbits in $C$ are
at most one-dimensional).
\end{prop}
\begin{proof} We first note that by the Krein-Milman theorem $\mu(M)$
is the convex hull of its extremal points, i.e.\ those points
$\alpha\in \mu(M)$ such that whenever $\alpha = (1-t)\beta + t\gamma$
for some $0<t<1$ and $\beta,\gamma\in \mu(M)$, then
$\beta=\gamma=\alpha$.

On the other hand, any point in $\mu(M)$ contains a neighborhood in
$\mu(M)$ which looks like the product of a $k$-plane with a cone --
this follows from the normal form for symplectic moment maps and the
fact that by Lemma \ref{lemmaPsi} the map $\Psi$ sends subsets of this
form to subsets of this form. This implies that for any extremal point
$\alpha=\mu(p)$ of $\mu(M)$ we have $d\mu_p=0$. Then one has, by the
definition of the twisted moment map,
\[ \iota_{\bar X} \omega_p = - \mu^X(p)
\theta_p=-\mu^X(p)\iota_V\omega_p
\] for all $X\in \lie t$; $V$ denotes the {\em s}-Lee vector field. This
means that $\overline{X}_p=-\mu^X(p)V$, hence the $T$-orbit through
$p$ coincides with the $V$-orbit. Thus $\mu(M)$ is the convex hull of
$\mu(C)$.

It remains to show that $\mu(M)$ has only finitely many extremal
points. For this we note that $C$ is the union of the (possibly empty)
$T$-fixed point set $M^T$ with the fixed point sets $M^H$, where $H$
runs through the (finitely many) codimension one isotropy subgroups
whose Lie algebra $\lie h$ does not contain $V$. Each of these fixed
point sets is a finite union of submanifolds of $M$. 

We first claim that $\mu$ sends every component $N$ of the $T$-fixed
point set to a single point. Indeed, for any $p\in N$ we have
$\theta_p = 0$ because the {\em s}-Lee vector field vanishes along $N$, and
hence $d\mu_p=0$.

For any codimension one subgroup with $V\notin \lie h$ we observe that
the minimal covering $\pi:\hat M\to M$ restricts, by equivariance of
$\pi$, to a covering $\pi:{\hat M}^H \to M^H$. We fix an arbitrary
component $N$ of $M^H$, as well as a component $\hat N$ of ${\hat
M}^H$ that is sent by $\pi$ to $N$. Consider the symplectic moment map
$\hat\mu:\hat M\to \lie t^*$ of the lifted $T$-action on the minimal
covering $\hat M$. We know that the image of $d \hat \mu_q:T_q\hat
M\to \lie t^*$, for all $q\in \hat N$, is contained in the annihilator
of $\lie h$, which is a one-dimensional subspace of $\lie t^*$. In
particular, $\hat\mu(\hat N)$ is a connected, closed subspace of a
straight line in $\lie t^*$. Using \eqref{eq:relPsi} 
\[ \mu\circ \pi = \Psi \circ \hat \mu
\] and that $\Psi$ sends segments to segments by Lemma \ref{lemmaPsi}
we conclude that $\mu(N)$ is a connected, closed subspace of a
straight line in $\lie t^*$ as well. It follows that $\mu(N)$ can
contain at most two extremal points of $\mu(M)$. As there are in total
only finitely many components $N$ of relevant fixed point submanifolds
$M^H$, we conclude that $\mu(M)$ has only finitely many extremal
points, hence is a convex polytope.
\end{proof}
 In the proof we observed that the components of $M^T$ (if
there exist any) are sent by $\mu$ to single points. The same holds
for those components $N$ of $M^H$ that are strict lcs submanifolds.
\begin{prop}\label{muconst} Let $T$ be a torus acting in twisted Hamiltonian fashion of Lee type on the lcs manifold $M$, and let $H\subset T$ be a closed subgroup of $T$, such that $\lie h\oplus \R V=\lie t$, and let $N$ be a connected component of $M^H:=\{x\in M\ |\ H.x=x\}$ such that $\theta|_N$ is not exact (we say that $N$ is strict lcs). Then $\mu|_N$ is constant.\end{prop}
Note that, from Lemma \ref{lemcompl} below, we see that $\omega|_N$ is nondegenerate, which means that $(N,\omega|_N)$ is lcs with Lee form $\theta|_N$. It is thus equivalent to say that $N$ is strict lcs and that $\theta|_N$ is not exact.
\begin{proof} Let us
show that $\left.d\mu_p\right|_{T_pN}:T_pN\to \lie t^*$ is the zero map for all $p\in N$.

For the {\em s}-Lee vector field $V$ we have $d\mu^V_p=0$ because $\mu^V\equiv 1$ by Remark \ref{twHam1}. For $X\in \lie h$
we first claim that $\mu^X(p)=0$. 
Indeed, on $N$ we have the equality $d\mu^X = \mu^X \theta$. If $N$ is
strict lcs, then the restriction (pullback) of $\theta$ on $N$ is not
exact, thus Lemma \ref{lem2.1} shows that $\mu^X=0$ on $N$ for
all $X\in\lie h$.

Therefore, if $p\in N$, then $d\mu_p(X)=0$, for all $X\in \lie t$, in particular $\mu|_N$ is constant.\end{proof}
\begin{corol}\label{Nstrictlcs}
With the notations above, $N$ is strict lcs if $\theta$ does not
vanish on $N$ (here, we mean not just $\theta|_{T_pN}\ne 0$, but
$\theta_p\ne 0$ for all $p\in N$).\end{corol}\begin{proof}
Suppose $N$ is gcs, thus
$\theta|_N=d\beta$, for $\beta:N\ra \R$. As $N$ is compact, $\beta$
has at least an extremal point $q$ such that $\theta_q=0$ on
$T_qN$.

On the other hand, we have the following Lemma.
\begin{lemma}\label{lemcompl}
Suppose $\omega$ is $T$-invariant, $H\subset T$ is a subgroup and let $q\in M^H$. Then $\theta_q$ is zero on a
complement of $T_qM^H$ in $T_qM$.\end{lemma}
\begin{proof}
Consider the decomposition of $T_pM$ into irreducible submodules of
the isotropy representation of the identity component of $H$: $T_pM =
T_pM^H \oplus \bigoplus_{i=1}^r W_i$, where $\lie h$ acts nontrivially
on the two-dimensional submodules $W_i$.  Note also that 
\[ 
T_pN = \{v\in T_pM\mid [\bar X,v]_p=0 \text{ for all } X\in \lie
h\}.\]
 Note that, because $\bar X$ vanishes at $p$, the bracket with $\bar
 X$ at $p$ induces a well-defined linear endomorphism of $T_pM$.
 Then, any $w\in W_i$ can be written as $w = [X,u]_p$ for some $X\in
 \lie h$ and $u\in W_i$, and because $\L_{\bar X} \omega = \theta({\bar
X})\omega = 0$ along $N$, we compute
 \[
 \theta(w) = \omega(V,w) = \omega(V,[\bar X,u]) = -\omega([\bar
 X,V],u) = 0, \]
where we have used the $T$-invariance of $\omega$.
\end{proof}
 Thus $\theta_q=0$ on $T_qM$, which is a
 contradiction, as the {\em s}-Lee vector field by assumption does not
 vanish at $q$.\end{proof}
Therefore, we have a clear description of the extremal points of $\mu(M)$, in case the Lee form is non-vanishing.
\begin{corol}\label{theta-not0}Let $M,T$ as above. If the Lee form $\theta$ does has no zeros, then the extremal points of $\mu(M)$ are $\mu(N)$, for $N$ a connected component of $M^H$, for $H$ a codimension one subtorus of $T$ (that does not  contain the flow of the {\em s}-Lee field $V$).\end{corol}
At the moment, we have no knowledge of an lcs manifold, admitting a twisted Hamiltonian torus action of Lee type, whose Lee form has zeros.
\subsection{The image of the symplectic moment map} Let a compact
torus $T$ act in twisted Hamiltonian fashion on a compact
lcs manifold $(M,\omega)$. Recall that
$\hat\mu:\hat M\ra \lie t^*$ is the symplectic moment map of $T$
(acting symplectically on $\hat M$, cf. Proposition \ref{liftG}), and
that this moment map does not depend on the choice of $\omega$ in its
conformal class, moreover it is $T$-equivariant (Proposition
\ref{eqiv}), which means it is constant on $T$-orbits.    

This symplectic moment map is thus an object of great interest, and it
would be interesting to prove a convexity result for it. As we show
below in Example \ref{ex:istrati}, its image is not always convex, so
additional conditions need to be considered for such a result:
\begin{thm}\label{convx-hat} Let $T$ be a compact torus acting in
  twisted Hamiltonian fashion on a compact, connected lcs manifold
  $(M,\omega)$ of rank $1$,
  such that the symplectic moment map $\hat\mu:\hat M\ra\lie t^*$ has
  no zero (in particular if the action is of Lee type for some
  conformally equivalent $\omega'=e^f\omega$). 

Then $\hat\mu$ is
  convex, its image $\hat\mu(\hat M)\subset \lie t^*\smallsetminus 0$
  is closed, and a convex cone. 

If there exists $\omega'=e^f\omega$
  such that the action of $T$ on $(M,\omega')$ is of Lee type, then
  $\hat\mu(\hat M)$ is the cone over the convex polytope $\mu(M)$, where $\mu$ is the twisted moment map of the $T$-action on $(M,\omega')$.
\end{thm}
\begin{proof} As noticed above, the difficulty of applying the
  local-global principle (Theorem \ref{thm:locglob}) to $\hat\mu$ is the fact
  that this map is usually not proper. We show, however, the following.
\begin{lemma}\label{hat-proper}The symplectic moment map $\hat\mu:\hat
  M\ra\R$  of a twisted Hamiltonian action of a compact Lie group $G$
  on a compact, connected lcs manifold
  $(M,\omega)$ is proper if and only if the lcs rank
  of $M$ is $1$ and $\hat\mu$ (or, equivalently, $\mu$) has no zero.
\end{lemma}\begin{proof}
Suppose the lcs rank is $1$, and $\mu(M)\subset \lie
g^*\smallsetminus 0$.
 Let $(x_n)_{n\in\N}$ be sequence in $\hat M$ such that
$(y_n)_{n\in\N}$, for $y_n:=\hat\mu(x_n)$ is convergent in $\lie
g^*$. If we denote by $a_n:=e^{-\lambda(x_n)}$ and
$z_n:=\mu(\pi(x_n))$, for $\pi:\hat M\ra M$,
then \eqref{mu0mu} implies that $y_n=a_nz_n$. 

Without loss of
generality we can assume that $(\pi(x_n))_{n\in\N}$ is convergent, thus
$z_n\ra z_0\in\lie g^*\smallsetminus 0$. This implies that $a_n$
converges as well. Now, using the properness of the map $\lambda$
(Proposition \ref{rk1}), we conclude that the sequence $x_n$ admits
a converging subsequence, which proves the properness of $\hat\mu$.

Conversely, if $\hat\mu$ is proper, then we choose the sequence
$(x_n)_{n\in\N}$ above such that $\lambda(x_n)$ converges. Thus
$(a_n)_{n\in\N}$ converges to $a_0>0$. Moreover,
since $M$ is compact, we assume $\pi(x_n)$, thus also $(z_n)$, are
convergent sequences. It follows that $y_n$ is convergent thus the
properness of $\hat\mu$ implies that $(x_n)_{n\in\N}$ admits a
convergent subsequence, thus $\lambda$ is proper.  By Proposition \ref{rk1}, the lcs rank of $M$ is $1$.

Note that if $\hat\mu$ (thus $\mu$) has a zero, then there exists a
sequence $(x_n)_{n\in\N}$ such that $\mu(x_n)\ne 0$ and $\mu(x_n)\ra
0$ (the twisted moment map $\mu$ cannot be constant zero, otherwise
all fundamental fields vanish everywhere - contradiction). By possibly
replacing $x_n$ by $\gamma_n(x_n)$, with $\gamma_n\in\Gamma$ such that
$$\lim_{n\ra\infty}\lambda(x_n)+\rho(\gamma_n)= +\infty,$$
we may assume that $\lambda(x_n)\ra +\infty$, which means that $x_n$
cannot have any convergent subsequence. This, combined with
$\hat\mu(x_n)\ra 0$, contradicts the properness of $\hat\mu$. 
\end{proof}
In the lcs rank $1$ case we can apply the local-global principle (Theorem
\ref{thm:locglob}) to the proper map $\hat\mu$ (which is locally open and locally convex by Lemma \ref{lemmastep1}), and we conclude that it is convex. 

From the convexity of $\hat\mu$ we obtain that its image is convex,
i.e. it contains the rays $\R_+^*\hat\mu(x)$, for $x\in\hat M$:
indeed, if $\gamma$ is a generator of $\Gamma$,
$\hat\mu(\gamma^n(x))=e^{-n\rho(\gamma)} \hat\mu(x)\in\hat\mu(\hat
M)$. The image of $\hat\mu$ is thus a cone over a compact set,
therefore it is closed in $\lie t^*\smallsetminus 0$. 

Finally, if the $T$-action is of Lee type for $\omega'=e^F\omega$, the
conclusion follows from Proposition \ref{prop:imagepolytope}.
\end{proof}
For higher lcs rank, we know that $\hat\mu(\hat M)\subset
\R^*_+\mu(M)$ is included in the cone over the compact set $\mu(M)$,
and it is dense in this set: as $\rho(\Gamma)\simeq \Z^k$, with $k>1$,
this subgroup is dense in $\R$, thus
$$\{\hat\mu(\gamma(x))\ |\ \gamma\in\Gamma\}$$
is dense in the ray $\R_+^*\mu(x)$. In general, we do not know whether
the image of $\hat\mu$ is a cone. We have, however, the following result.
\begin{thm}\label{conemoment}
  Let $T$ act in twisted Hamiltonian fashion on the compact, connected lcs manifold $M$ of rank $k\ge 2$, such that the Lee form is nowhere-vanishing. Suppose the symplectic moment map (or equivalently, the twisted moment map) has no zero. Then the image of the symplectic moment map $\hat\mu:\hat M\ra\lie t^*$ is a cone without the apex. If there additionally exists a $2$-form $\omega$ for which the $T$-action is of Lee type, this cone is the cone over $\mu(M)$ and is thus convex. \end{thm}\begin{proof}
  Note that in all regular points of the $T$-action on $\hat M$ (denote by $M_0$ the open set of these points), the moment
map is locally open to its image in $\lie t^*$, so by the density argument above, this implies that $\hat\mu(M_0)$ is a cone.

Consider now $p$ a  point with nontrivial isotropy $H$. Then $M^H$ is a finite union of strict lcs submanifolds $N$
(as in Corollary \ref{Nstrictlcs}, but with arbitrary dimension of $H$). Then $\mu^X$ is
zero along $N$, for all $X\in\lie h$, with the same argument as in Proposition \ref{muconst}. This implies that $\hat\mu^X$ is zero along ${\hat M}^H$. Now, for the
symplectic moment map, the image of $d\hat\mu_p$ is the annihilator of the isotropy subalgebra $\lie t_p=\lie h$. This means: we have shown that $\hat\mu(p)$ itself is in the
annihilator of $\lie h$, and thus $\hat\mu(p)$ is in the image of $d\hat\mu_p$.

Therefore, if $\hat\mu(p)\ne 0$, then the vector $\hat\mu(p)$ lies in the image of $d\hat\mu_p$, and thus  a small
interval around $\hat\mu(p)$, in radial direction, is in the image of $\hat\mu$. Combined with the density property above, this shows that the full radial line $\R^*_+\hat\mu(p)$ is in the image of $\hat\mu$.


Clearly, if the action of $T$ is of Lee type with respect to $\omega$, then $\hat\mu(\hat M)$, which is a cone itself, and is included in the cone over $\mu(M)$ as a dense subset in it, must coincide with $\R^*_+\cdot\mu(M)$.
\end{proof}
Note that it is not necessary that the Lee form $\theta$, for which the action is of Lee type, is nowhere vanishing. It is enough that {\em there exists} a nowhere vanishing form $\theta'$ in the cohomology class of $\theta$, and that the action has no fixed points.  
\begin{rem} It is always possible to obtain a twisted Hamiltonian
  torus action such that the twisted moment map has zeroes: take an
  action where $\hat \mu(M)\ni 0$ and consider a subtorus $T'\subset T$
  such that the annihilator of $\lie t'$ in $\lie t^*$ lies in
  $\hat \mu(M)$. Then the twisted moment map of the action of $T'$ is
  $\mu'=q\circ\mu$ (where $q:\lie t^*\ra {\lie t'}^*$ is the dual of
  the inclusion $\lie t'\hookrightarrow\lie t$) and it has zeroes.
\end{rem}

\section{The twisted moment map of Vaisman manifolds} \label{sec:momentmapvaisman} In this section we consider an
effective, holomorphic and twisted Hamiltonian action of a compact
connected Lie group $G$ on a compact Vaisman manifold $(M, g, J,
\theta)$, with twisted moment map $\mu \colon M \rightarrow \lie
t^*$. 

In the Vaisman setting, such an action is automatically isometric and
every field induced by $\lie g$ lies in the kernel of the Lee form,
because a Vaisman metric is Gauduchon (cf. comments at the end of
Section 2, see also~\cite[Lemma~4.2]{Pilca}).

We assume without loss of generality that the Lee vector field is of
unit length. Let $S$ be the associated Sasakian manifold, as a
codimension one submanifold of the minimal covering $\hat M$ of
$M$. In Section \ref{sec:presentations} it is explained that $\hat M$
is the cone~$C(S)$.

The explicit expression of the twisted moment map on $M$ is, for $X
\in \lie g$, given by Proposition \ref{dexact} (see also~\cite[Lemma~4.4]{Pilca} with a different
sign convention)
\begin{equation}
\label{eq:momentVaisman} \mu^X = -\theta(J \overline{X}).
\end{equation}

Let $\lambda$ be the coordinate on the $\R$-factor on the minimal covering
$\hat M = \R \times S$. On the other hand, on $\hat M$ the lift of
$\theta$ is exactly $d\lambda$ and the K\"ahler form is exact and given by
$d(e^{-2h} \eta)$, where $\eta = -d^c h$ (see
e.g.~\cite{sparks_survey}). Note that $\eta$ is precisely the contact
form of the associated Sasakian manifold $S$. As shown in Lemma
\ref{liftG}, see also \cite{MaMoPi}, the $G$-action lifts to $\hat M$,
and the lifted action respects the Sasakian leaves. The \emph{contact
moment map} for the induced action on $S$ is
\[ X \mapsto \eta(\overline{X}) = -d^c t (\overline{X}) =
-\theta(J\overline{X}).
\] As the moment map $\mu$ is $\theta^\#$-invariant, the following follows.
\begin{prop}\label{prop:momentimageSasakian} The twisted Hamiltonian
action of a compact connected Lie group $G$ on a complete Vaisman
manifold by Vaisman automorphisms and the induced action on its
associated Sasakian manifold share the same moment image in~$\lie
g^*$.
\end{prop}

\begin{rem} By Proposition~\ref{prop:cptSasaki}, in the case of
Vaisman manifolds of lcK rank one, the associated Sasakian manifold is
compact. Thus, in this case the Convexity Theorem \ref{thm:convmu}
follows  trivially from known convexity results for compact Sasakian
manifolds and Proposition \ref{prop:momentimageSasakian}.
\end{rem}


For the remaining of this section, we focus on the relation between
the twisted moment image of $\mu \colon M \to \lie g^*$ and the
K\"ahler moment image of $\hat \mu \colon \hat M \to \lie g^*$.

Consider the vector field $X$ given by the gradient of $h$ with
respect to the gcK metric $\pi^* g$. This is nothing but the metric
dual with respect to $\pi^* g$ of the lift of the Lee form, and it
satisfies $\pi_*X=\theta^\#$. Let $\phi : \R \times \hat M \rightarrow
\hat M$ be the flow of $X$.

\begin{lemma} \label{muss} We have $\hat \mu^Z \circ \phi_s = e^{-s}
\hat \mu^Z$ for all $Z\in \lie g$.
\end{lemma}

\begin{proof} We have that $\mu$ is $\theta^\#$-invariant and, by
Equation \eqref{mu0mu}, that $\hat\mu = e^{-h}\mu\circ \pi$. Then we
compute
\begin{align*} X(\hat\mu^Z) &= X(e^{-h}\mu^Z\circ \pi)= X(e^{-h})
\mu^Z\circ \pi= -X(h) \hat\mu^Z= - \hat\mu^Z,
\end{align*} using that the Lee vector field has constant length $1$.
\end{proof}

The following result is, in fact, a particular case of Theorem \ref{conemoment}:

\begin{prop} \label{thm:img} We have $\hat\mu(\hat M) = {\mathbb
R}^{>0}\cdot \mu(M)$.
\end{prop}

\begin{proof} Consider the associated Sasakian manifold
$S=h^{-1}(0)$. Because $\hat\mu = e^{-h}\mu\circ \pi$ and $\mu$ is
$\theta^\#$-invariant, we have $\hat\mu(S) = \mu(\pi(S)) =
\mu(M)$. Then Lemma \ref{muss} implies the claim. 
\end{proof}


A Vaisman manifold is \emph{strongly regular} if both $\theta^\#$ and
$J \theta^\#$ induce a free circle action. In this case we can consider the
projection space $\pi:M\to \check{M} = M/ \langle \theta^\#, J
\theta^\# \rangle$ to the orbit space of the associated $T^2$-action,
which is a K\"ahler manifold. As before, we consider the action of a
compact connected Lie group $G$ on $M$ by Vaisman automorphisms which
is twisted Hamiltonian, with moment map $\mu:M\to \lie g^*$. We have
then the following result, essentially contained in \cite{Pilca}.

\begin{prop} \label{thm:regularaction} Let $(M, g, J, \theta)$ be a
strongly regular Vaisman manifold with an effective twisted
Hamiltonian $G$-action by Vaisman automorphisms. Then the $G$-action
descends to the K\"ahler manifold $\check{M}$. It is Hamiltonian, with
a moment map $\check{\mu}:\check{M}\to \lie g^*$ defined by $\mu =
\check{\mu}\circ \pi$. In particular, the moment images satisfy
$\check\mu(\check M) = \mu(M)$.
\end{prop}

\begin{proof} As the action is by Vaisman automorphisms, the vector
fields $\theta^\#$ and $J \theta^\#$ are central in
$\lie{aut}(M,J,g)$. So the $G$-action descends to a K\"ahler action on
$\check{M}$.  

The moment map $\mu: M\to \lie g^*$ descends to a map $\check{\mu}:
\check{M} \rightarrow \bar{\lie g}^*$. Indeed, by the definition of
twisted moment maps and recalling Equation \eqref{eq:momentVaisman}
and that $\lie g \subseteq \ker \theta$, we can verify that it is
invariant under the action of $\theta^\#$ and $J \theta^\#$.

The proof that $\check{\mu}$ is a moment map for the $G$-action on
$\check{M}$ is done along the same line of the first part of the proof
of Theorem~5.1 of \cite{Pilca}.
\end{proof}

\begin{rem}If the action is additionally of Lee type, then the
subgroup spanned by (the element that induces) $J \theta^\#$ is a
subgroup $H\subset G$ isomorphic to $S^1$, which is central in $G$ and
acts trivially on $\check{M}$. We define $\check{G} = G/H$, and
consider the projection $p:G\to \check{G}$.

For $Z\in \lie h$ we have $d_\theta\mu^Z = i_{\bar Z} \omega = 0$,
because $H$ acts trivially on $\check{M}$. As $d_\theta$, by Lemma
\ref{lem:twisteddiffinj}, is injective on functions, this implies that
$\mu^Z=0$.  It follows that we can define a map
$\tilde\mu:\check{M}\to {\check{g}}^*$ such that $p^*\circ \tilde\mu =
\check{\mu}$; it is a moment map for the $\check{G}$-action on
$\check{M}$. Then the moment images are related by
\[ p^*(\tilde\mu(\check{M})) = \check{\mu}(\check{M}) = \mu(M).
\]
\end{rem}

\subsection{Toric Vaisman manifolds} \label{subsec:toric}

In this subsection, we focus on \emph{toric} Vaisman manifolds.
\begin{defi} \label{def:toric} A connected, compact lcs, resp. lcK
manifold of real dimension $2n$ is \emph{toric} if it admits an
effective (for lcK also holomorphic) and twisted Hamiltonian action of
the $n$-dimensional compact torus.
\end{defi}

\begin{rem} In \cite{istrati_toric} it is shown that toric lcK
manifolds admit a toric Vaisman structure, and an example of a toric
lcs manifold is given which does not admit a toric lcK structure. We
consider this example below in Example \ref{ex:istrati}.
\end{rem}

Let $M$ be a toric Vaisman manifold, acted on by a compact torus
$T$. In \cite[Prop.\ 5.4]{MaMoPi}, it is proven that $M$ is the
mapping torus of its associated toric Sasakian manifold $S$ with
respect to a $T$-equivariant Sasakian automorphism\footnote{The
Riemannian geometric part of this statement follows from the classical
result in \cite[Chap.~VI, Sec.~3, Lemma~2]{kn}, that was extended to
the classification of homotheties of a Riemannian cone in
\cite[Thm.~5.1]{GOOP}. The fact that it is a contact transformation is
equivalent to the cone homothety being holomorphic and follows from
the well-known relations between contact structures and their
symplectizations, see e.g. \cite[Chap.~6]{Blair}.} $\psi$ such that
the map $f(t, p) = (t+\alpha, \psi(p))$ generates the $\Z$-action on
the K\"ahler cone $\R \times  S$. The purpose of this section is to
show that the Sasakian automorphism is necessarily given by an element
of $T$.

Let $S$ be a $(2n+1)$-dimensional toric Sasakian manifold acted by a
$(n+1)$-~dimensional torus $T$. We denote a moment map of the induced
action on the K\"ahler cone by $\mu\colon C(S) \rightarrow
\R^{n+1}$. Let $\call C$ be the image of the moment map of the
K\"ahler cone $C(S)$. Then the open set
\[ C(S)^0 = \mu^{-1}(\call C^0) = \{ (t,p) \in C(S)\mid T_{(t,p)} = 1
\},
\] consists of the points in which the $T$-action is free.

Following Abreu \cite{abreu_cones}, $C(S)^0$ is described by
\emph{action-angle coordinates} $x = \mu \in \call C^0$ and $y \in
T$. On $C(S)^0$, the $T$-action is given by torus translations on the
second component.

Moreover, the moment map has the form $\mu^X (t,p) = e^{-2t} \eta_p
(\bar X_p)$, where $\eta$ is the contact form of $S$ and by the bar we
mean the induced field on $S$.

\begin{lemma} \label{lemma:lemmaT} For any $\alpha \in \R$, the
transformation $f(t,p) = (t+\alpha, \psi(p))$ is a $T$-equivariant
holomorphic homothety of the toric K\"ahler cone $C(S)$ if and only if
it is induced by an element of $T$.
\end{lemma}
\begin{proof} In the $(x,y)$-coordinates, the K\"ahler form on
$C(S)^0$ has the form
	\begin{equation} \omega_{(x,y)} = \begin{pmatrix} 0 & -I \\ I
& 0
	\end{pmatrix}
	\end{equation} and the complex structure is
	\begin{equation} J_{(x,y)} = \begin{pmatrix} 0 & -S(x)^{-1} \\
S(x) & 0
	\end{pmatrix}
	\end{equation} where $S(x)$ is the Hessian of a smooth
function $s\colon \call C^0  \to \R$. 
degree $-1$, i.e. $S(e^{2t}x) = e^{-2t} S(x)$ for all $t \in \R$ and
$x \in \call C^0$.
	
	Being $f(t,p) = (t+\alpha, \psi(p))$, that is a translation on
the $\R$-factor and a $T$-equivariant contactomorphism on $S$, it is
clear that it satisfies $\mu \circ f = e^{-2\alpha} \mu $, so in
action-angle coordinates it has the form
	\begin{equation} f(x,y) = (e^{-2\alpha} x, f_2(x,y)).
	\end{equation}
	
	Imposing $T$-equivariance, we have the condition $f_2(x,
y+y_0) = f_2(x, y) + y_0$ and we obtain, by differentiation, that 
	\begin{equation} \label{eq:def2} \frac{\de f_2}{\de y} (x,y) =
I.
	\end{equation} Note that here we have identified $f_2 \colon
\call C^0 \times T \to T$ with its lift $f_2 \colon \call C^0 \times
\R^n \to \R^n$ and that the plus sign denotes addition in $\R^n$.
	
	Denote $A(x,y) = \frac{\de f_2}{\de x} (x,y)$. Swapping
partial derivatives in \eqref{eq:def2} we obtain that $A$ does not
depend on $y$. Hence the differential of $f$ has the form
	\begin{equation} df_{(x,y)} = \begin{pmatrix} e^{-2\alpha} I &
0 \\ A(x) & I
	\end{pmatrix}.
	\end{equation}
	
	The condition that $f$ is holomorphic, i.e. $df_{(x,y)}
J_{(x,y)} = J_{f(x,y)} df_{(x,y)}$ translates into
	\begin{align*}
 \begin{pmatrix} e^{-2\alpha} I & 0 \\ A(x) & I
	\end{pmatrix} \cdot \begin{pmatrix} 0 & -S(x)^{-1} \\ S(x) & 0
	\end{pmatrix} = \begin{pmatrix} 0 & -S(e^{-2\alpha} x)^{-1} \\
S(e^{-2\alpha} x) & 0
	\end{pmatrix} \cdot \begin{pmatrix} e^{-2\alpha} I & 0 \\ A(x)
& I
	\end{pmatrix}
	\end{align*} that implies $A=0$. Note that $df$ is a homothety
of $\omega$ with factor $e^{-2\alpha}$.
	
	Hence, on $C(S)^0$, the map $f$ has the form $f(x,y) = (e^{-2
\alpha} x, y+y_0)$ for some $y_0 \in T$. So, $f$ coincides on $C(S)^0
\cap S$ with the transformation induced by the action of the torus
element $y_0$.
\end{proof}

We can thus infer the following.

\begin{lemma} \label{lemma:equivisotopic} Any equivariant Sasakian
automorphism of a toric Sasakian manifold is equivariantly isotopic to
the identity.
\end{lemma}
\begin{proof} Given such an automorphism $\psi$, extend it to a
homothety $f(t,p) = (t+\alpha, \psi(p))$ on the K\"ahler cone.
	
	By Lemma \ref{lemma:lemmaT}, $\psi$ is defined by the action
of an element $\sigma_\psi \in T$. Then the map $\psi_t (x) =
\sigma(t)$ is a $T$-equivariant isotopy between $\psi$ and the
identity, where $\sigma(t)$ is any smooth curve joining $\sigma$ and
$1$ in $T$.
\end{proof}

\begin{thm} \label{thm:prod}Let $M$ be a compact toric Vaisman
manifold. Then the following hold.
\begin{enumerate}
\item \label{mappingtorus} $M$ is isomorphic to the mapping torus of a
toric Sasakian manifold $S$ by an element $g\in T$: it is $M =
S_{\alpha,g}:=  \R \times S/(t,p)\sim (t+\alpha,gp)$ with the natural
Vaisman structure.

\item \label{diffeoprod} $M$ is diffeomorphic, as a smooth manifold,
to the direct product of a toric Sasakian manifold and a circle.

\item \label{uniqueness} Two toric Vaisman manifolds $S_{\alpha,g}$
and $W_{\beta,h}$ are $T$-equivariantly isomorphic as Vaisman
manifolds if and only if $S=W$, $\alpha = \beta$ and $g=h$.
\end{enumerate}
\end{thm}

\begin{proof} The first statement is \cite[Prop.\ 5.4]{MaMoPi},
together with Lemma~\ref{lemma:lemmaT}. The second one is Lemma
\ref{lemma:equivisotopic}. For the last one, observe first that an
isomorphism of Vaisman manifolds induces an isomorphism of the
associated Sasakian manifolds. The number $\alpha$ is just the length
of the flow of the Lee vector field, so in particular an invariant of
the Vaisman structure. Now, considering two Vaisman manifolds
$S_{\alpha,g}$ and $S_{\alpha,h}$, any equivariant isomorphism $f$
lifts to an equivariant holomorphic homothety of the associated toric
K\"ahler cone $C(S)$. Then, by Lemma \ref{lemma:lemmaT}, there exists
a $\sigma \in T$ such that $F$ is induced by $(t,p) \mapsto (t+\gamma,
\sigma(p))$, for some $\gamma \in \R$. The K\"ahler cone
transformations defined by $g$ and $h$ are conjugated under $F$, so in
$T$, the elements $g$ and $h$ are conjugated under $\sigma$. But, $T$
being Abelian, it follows that $g=h$. 
\end{proof}

\begin{rem} \label{rmk:modulispaceToricVaisman} It is well-known, see
e.g.\ \cite{FOW} or \cite{martelli_sparks_yau}, that the choice of a
strongly convex polyhedral cone $\call C$ together with a point $\xi$
in its interior $\call C^0$ define a toric Sasakian manifold and that
there is no uniqueness in this construction. All toric Sasakian
manifolds with given Reeb field and moment cone are parameterized by
$\call C_0^* \times \call H(1)$, where $C_0^*$ is the interior of the
dual cone $\call C^*$ and $\call H(1)$ is the space of smooth
homogeneous degree one functions on $\call C$, such that their Hessian
is strictly convex. Then the theorem shows that the collection of
toric Vaisman manifolds, up to $T$-equivariant Vaisman isomorphism,
with moment image $\call C \cap \chi_\xi$ is in a one-to-one
correspondence with $\call C_0^* \times \call H(1)\times T$.
\end{rem}

\begin{rem} \label{rmk:topologytoricVaisman} Part~\eqref{diffeoprod}
of Theorem~\ref{thm:prod} can give information about the topology of a
compact toric Vaisman manifold.

Combining Theorem \ref{thm:prod} with results about
the topology of toric contact manifolds (see e.g. \cite{monoBG,
LermanHomotopy}) we can easily obtain information about the topology
of a toric Vaisman manifold $M$. In particular, we reobtain the fact
that $b_1(M)=1$, already proven in \cite{MaMoPi}, and one can compute
the first and second homotopy group of $M$ in terms of the
combinatorial data of the moment cone. Moreover, in \cite{LiFundGroup}
it is proved that the fundamental group of a toric contact manifold of
Reeb type is finite cyclic of order, say, $k$ and it is explained how
to read $k$ from the moment data.
\end{rem}



\section{Examples} \label{sec:examples} In this section we provide
examples of twisted Hamiltonian actions on various lcs and lcK
structures. We begin with the most standard example of a Vaisman
manifold.

 \begin{example}[The weighted Hopf manifold] \label{ex:weightedHopf}
Consider the action on $\C^n \setminus \{ 0\}$ given by
\[ \gamma(z) = (a_1 z_1, \ldots, a_n z_n),
\]	 for some complex numbers $a_j$ such that $0 < |a_j| < 1$ and
let $\Gamma$ be the subgroup of holomorphic isometries of  $\C^n
\setminus \{ 0\}$ generated by $\gamma$.

In \cite[Thm.~B]{KamOrnea}, it is proved that  $M=(\C^n \setminus \{
0\})/\Gamma$ carries a Vaisman structure, whose associated Sasakian
manifold is the weighted Sasakian sphere $S^{2n-1}_w$ with weights
$w_j = \log |a_j|$, Reeb field
\begin{equation}
\label{weightedreeb} \xi_w = \sum_{j=1}^n w_j \biggl( x_j
\frac{\de}{\de y_j} - y_j \frac{\de}{\de x_j} \biggr)
\end{equation} and contact form
\begin{equation}
\label{weightedeta} \eta_w = \frac{1}{\sum_j w_j |z_j|^2} \eta_0,
\end{equation} where $\eta_0$ is the standard contact form on
$S^{2n-1}$.

On both $M$ and $S^{2n-1}_w$, the standard torus $T^n$ is taken with
its standard action. The image of the twisted moment map of $M$ is, by
Theorem \ref{thm:img}, the same as the contact moment map, that is
given by the intersection of the first octant $\{ x \in \R^n\mid x_j >
0\}$ with the characteristic hyperplane $\{ x\mid \langle x, w \rangle
=1 \}$.

The construction of \cite{KamOrnea} generalizes the one of
\cite{gaudornea} for $n=2$. Moreover, in \cite{MaMoPi}, it is proven
that this weighted Hopf manifold, for $n=2$, is the only toric Vaisman surface. 
\end{example}

\begin{example} \label{ex:momentimageconfchange} Changing conformally
the form $\omega$ leads to a corresponding change of the twisted
moment map, see Remark \ref{r2.10}. We give here an example where the
process leads to a twisted moment map with non-convex image (the Lee
type condition is lost as well):  
We start with
the standard Vaisman structure $(g, J)$ on a diagonal Hopf surface
$M$, i.e., we let $n=2$ and $a_1=a_2$ in Example \ref{ex:weightedHopf}. Let $\mu \colon M \to \R^2$ be the twisted moment map of the
standard $\T^2$-action which is of Lee type, given by
\[ \mu(z) = \frac{1}{\|z\|^2} (|z_1|^2, |z_2|^2).
\]

Let $f \colon M \to \R$ be a smooth function of the variable
$t=|z_1|^2/\|z\|^2$. The new lcK structure $(e^{-f}g, J)$ is then
still lcK and acted upon by the same two-dimensional torus, with
twisted moment map $\tilde \mu := e^{-f} \mu \colon M \to \R^2$. The
action on the new structure is not of Lee type.

The image of $\tilde \mu$ is
\[ e^{-f} \mu(M) = \{ e^{-f(t)}(t,1-t)\mid t \geq 0 \}
\] and we note that it is a convex subset of the plane if and only if
it is a piece of a straight line, which happens exactly when $f$ is
constant.
\end{example}

\begin{example}\label{ex:istrati} The following example was given in
\cite[Ex.~6.4]{istrati_toric}, as a toric lcs manifold that does not
admit a toric lcK metric. Let the standard $2$-torus $T=\T^2$ act on
the compact manifold $M = (S^1)^4$ by multiplication on the first two
factors
\[ (e^{it_1}, e^{it_2}) \cdot (e^{i \theta_1}, e^{i \theta_2},e^{i
\theta_3},e^{i \theta_4}) = (e^{i (t_1+\theta_1)}, e^{i(t_2+
\theta_2)},e^{i \theta_3},e^{i \theta_4}).
\] Let $d \theta_i$ be the volume form on the $i$-th $S^1$-factor of
$M$ and consider the forms
\begin{align*} \theta 	&:= d\theta_4 \\ \eta 	&:= \sin \theta_3
d\theta_1 + \cos \theta_3 d \theta_2
\end{align*} on $M$. The form $\omega := d_\theta \eta$ is a
$T$-invariant lcs form on $M$ and the form $-\eta$ defines a twisted
moment map $\mu \colon M \to \lie t^*$ of the $T$-action.

For $Y = (y_1, y_2) \in \R^2 \cong \lie t$, the induced field at $x =
(e^{i \theta_1}, e^{i \theta_2},e^{i \theta_3},e^{i \theta_4})$ is
\begin{equation} \label{eq:IstratiInducedField} \bar Y_x = y_1
\frac{\de}{\de \theta_1} + y_2 \frac{\de}{\de \theta_2}
\end{equation} hence $\eta(\bar Y_x) = -y_1 \sin \theta_3 - y_2 \cos
\theta_3$, and
\[ \mu^Y(x) = - \eta( \bar Y_x) = y_1 \sin \theta_3 + y_2 \cos
\theta_3.
\] The image of the twisted moment map $M \to \lie t^* \cong \R^2$
thus is a (non-convex) circle in the plane.

This does not contradict Theorem \ref{thm:convmu} because the action
is not of Lee type: the {\em s}-Lee vector field of the lcs structure is
$V = \sin \theta_3 \frac{\de}{\de \theta_1} + \cos \theta_3
\frac{\de}{\de \theta_2}$ and we can see it cannot be induced by any
field of the form~\eqref{eq:IstratiInducedField}.

Note that in this example is not possible to perform a conformal
change in order to make the moment image convex, cf.\ Example
\ref{ex:momentimageconfchange}. Note also that the minimal covering of
$M$ is $\hat M = (S^1)^3\times \R$. The moment map of the lifted
$T$-action is $\hat \mu = e^{-\theta_4}\mu:\hat M \to \lie t^*$ and
has nonconvex image $\lie t^*\setminus \{0\}$.
\end{example}

We conclude with a lcK non-Vaisman example that satisfies the
assumptions of our Theorem~\ref{thm:convmu}.

\begin{example} Let $(\omega, \theta)$ be a Vaisman structure on a
complex manifold $(M, J)$. In \cite[Section~3]{Mor2Ornea}, the authors
start with a Vaisman structure with Lee field $V$ and deform it, in a
non-conformal way, to a non-Vaisman lcK structure with the same Lee
field. They use a non-constant function $f \colon M \to \R$ with $f >
-1$ and gradient collinear with $V$.

In detail, they take the new lcK form $\bar \omega = \omega + f \theta
\wedge J\theta$, with Lee form $\bar \theta = (1+f) \theta$. We claim
that if $(M, J, \omega, \theta)$ admits a holomorphic twisted
Hamiltonian action of Lee type of a Lie group $G$, then the same
action is still twisted Hamiltonian with respect to the new lcK
structure, with the same moment map.

Let $X \in \lie g$. The twisted moment map of the $G$-action, with
respect to the Vaisman structure, is (see \eqref{eq:momentVaisman})
given by $\mu^X = \iota_X J \theta = -\theta(J\bar X)$, and we note
that $\bar X \in \ker \theta$, by \cite[Lemma~4.2]{Pilca}. We compute
\begin{align*} \iota_X \bar \omega &= \iota_X \omega + f \iota_X
(\theta \wedge J \theta) \\ &= \iota_X \omega - f \theta \wedge
\iota_X J \theta \\ &= d_\theta \mu^X - f \mu^X  \theta\\ &= d \mu^X -
\mu^X \theta - f \mu^X  \theta\\ &= d \mu^X  - (1+f) \mu^X \theta \\
&= d_{\bar \theta} \mu^X.
\end{align*}

The authors give an explicit example of such a function $f$ on the
complex $n$-dimensional Hopf manifold. In this way they obtain a toric
non-Vaisman lcK structure on the Hopf manifold. Its moment map is the
same as in the Vaisman case, see Example \ref{ex:weightedHopf}.
\end{example} \bibliography{../../allbib/allbib}

\providecommand{\bysame}{\leavevmode\hbox to3em{\hrulefill}\thinspace}
\providecommand{\MR}{\relax\ifhmode\unskip\space\fi MR }
\providecommand{\MRhref}[2]{%
  \href{http://www.ams.org/mathscinet-getitem?mr=#1}{#2}
}
\providecommand{\href}[2]{#2}
\begin{thebibliography}{10}

\bibitem{abreu_cones}
M.~Abreu, \emph{{K{\"a}hler-{S}asaki geometry of toric symplectic cones in
  action-angle coordinates}}, Port. Math. \textbf{67} (2010), no.~2, 121--153.
  \MR{2662864}

\bibitem{atiyah_conv}
M.~F. Atiyah, \emph{{Convexity and commuting {H}amiltonians}}, Bull. London
  Math. Soc. \textbf{14} (1982), no.~1, 1--15. \MR{642416}

\bibitem{BazzoniGoertsches_convexity}
G.~Bazzoni and O.~Goertsches, \emph{{Toric actions and convexity in
  cosymplectic geometry}}, arXiv:1712.08141.

\bibitem{belgun_surfaces}
F.~A. Belgun, \emph{{On the metric structure of non-{K}{\"a}hler complex
  surfaces}}, Math. Ann. \textbf{317} (2000), no.~1, 1--40. \MR{1760667}

\bibitem{BirteaOrtegaRatiu}
P.~Birtea, J.-P. Ortega, and T.~S. Ratiu, \emph{{Openness and convexity for
  momentum maps}}, Trans. Amer. Math. Soc. \textbf{361} (2009), no.~2,
  603--630. \MR{2452817}

\bibitem{BjoKar}
C.~Bjorndahl and Y.~Karshon, \emph{{Revisiting {T}ietze-{N}akajima: local and
  global convexity for maps}}, Canad. J. Math. \textbf{62} (2010), no.~5,
  975--993. \MR{2730351}

\bibitem{Blair}
D.~E. Blair, \emph{{Riemannian geometry of contact and symplectic manifolds}},
  second ed., {Progress in Mathematics}, vol. 203, Birkh{\"a}user Boston, Inc.,
  Boston, MA, 2010. \MR{2682326}

\bibitem{monoBG}
C.~Boyer and K.~Galicki, \emph{{Sasakian Geometry}}, Oxford Science
  Publications, 2007.

\bibitem{DO}
S.~Dragomir and L.~Ornea, \emph{{Locally Conformal {K}{\"a}hler Geometry}},
  {Progress in Mathematics}, Birkh{\"a}user Boston, 2012.

\bibitem{FOW}
A.~Futaki, H.~Ono, and G.~Wang, \emph{{Transverse {K}{\"a}hler geometry of
  {S}asaki manifolds and toric {S}asaki-{E}instein manifolds}}, J. Diff. Geom.
  \textbf{83} (2009), 585--635.

\bibitem{gaud}
P.~Gauduchon, \emph{{Calabi's extremal {K}{\"a}hler metrics: An elementary
  introduction}}, To appear.

\bibitem{gaudornea}
P.~Gauduchon and L.~Ornea, \emph{{Locally conformally {K}{\"a}hler metrics on
  {H}opf surfaces}}, Ann. Inst. Fourier (Grenoble) \textbf{48} (1998), no.~4,
  1107--1127. \MR{1656010 (2000g:53088)}

\bibitem{GOOP}
R.~Gini, L.~Ornea, M.~Parton, and P.~Piccinni, \emph{{Reduction of {V}aisman
  structures in complex and quaternionic geometry}}, J. Geom. Phys. \textbf{56}
  (2006), no.~12, 2501--2522. \MR{2252875}

\bibitem{convexity_bm}
V.~Guillemin, E.~Miranda, and J.~Weitsman, \emph{{Convexity of the moment map
  image for torus actions on $b^m$-symplectic manifolds}}, arXiv:1801.01097.

\bibitem{techniques}
V.~Guillemin and S.~Sternberg, \emph{{Symplectic techniques in physics}},
  Cambridge University Press, 1984.

\bibitem{HallerRybicki}
S.~Haller and T.~Rybicki, \emph{{Reduction for locally conformal symplectic
  manifolds}}, J. Geom. Phys. \textbf{37} (2001), no.~3, 262--271. \MR{1807280}

\bibitem{HilgertNeebPlank}
J.~Hilgert, K.~Neeb, and W.~Plank, \emph{{Symplectic convexity theorems}}, Sem.
  Sophus Lie \textbf{3} (1993), no.~2, 123--135. \MR{1270171}

\bibitem{istrati_toric}
N.~Istrati, \emph{{A characterisation of toric {LCK} manifolds}},
  arXiv:1612.03832, 2016.

\bibitem{is-ot}
N.~{Istrati} and A.~{Otiman}, \emph{{De {R}ham and twisted cohomology of
  {O}eljeklaus-{T}oma manifolds}}, arXiv:1711.07847, 2017.

\bibitem{KamOrnea}
Y.~Kamishima and L.~Ornea, \emph{{Geometric flow on compact locally conformally
  {K}{\"a}hler manifolds}}, Tohoku Math. J. (2) \textbf{57} (2005), no.~2,
  201--221. \MR{2137466}

\bibitem{kn}
S.~Kobayashi and K.~Nomizu, \emph{{Foundations of Differential Geometry}},
  Interscience Publisher, 1963.

\bibitem{lerman}
E.~Lerman, \emph{{A convexity theorem for torus actions on contact manifolds}},
  Illinois J. Math. \textbf{46} (2002), no.~1, 171--184. \MR{1936083}

\bibitem{LermanHomotopy}
\bysame, \emph{{Homotopy groups of {$K$}-contact toric manifolds}}, Trans.
  Amer. Math. Soc. \textbf{356} (2004), no.~10, 4075--4083. \MR{2058839}

\bibitem{LiFundGroup}
H.~Li, \emph{{The fundamental group of toric contact manifolds}},
  arXiv:1703.07277v2.

\bibitem{MaMoPi}
F.~Madani, A.~Moroianu, and M.~Pilca, \emph{{On toric locally conformally
  {K}{\"a}hler manifolds}}, Ann. Global Anal. Geom. \textbf{51} (2017), no.~4,
  401--417. \MR{3648998}

\bibitem{martelli_sparks_yau}
D.~Martelli, J.~Sparks, and S.-T. Yau, \emph{{The geometric dual of
  {$a$}-maximisation for toric {S}asaki-{E}instein manifolds}}, Comm. Math.
  Phys. \textbf{268} (2006), no.~1, 39--65. \MR{2249795}

\bibitem{Mor2Ornea}
A.~{Moroianu}, S.~{Moroianu}, and L.~{Ornea}, \emph{{Locally conformally
  {K}{\"a}hler manifolds with holomorphic {L}ee field}}, arXiv:1712.05821,
  2017.

\bibitem{NittaGeneralized}
Y.~Nitta, \emph{{Convexity properties of generalized moment maps}}, J. Math.
  Soc. Japan \textbf{61} (2009), no.~4, 1171--1204. \MR{2588508}

\bibitem{otiman}
A.~{Otiman}, \emph{{Locally conformally symplectic bundles}}, arXiv:1510.02770,
  to appear in J. Sympl. Geom., 2015.

\bibitem{PalaisTerng}
R.S. Palais and C.~Terng, \emph{{Critical Point Theory and Submanifold
  Geometry}}, {Lecture Notes in Mathematics}, Springer Berlin Heidelberg, 1988.

\bibitem{Pilca}
M.~Pilca, \emph{{Toric {V}aisman manifolds}}, J. Geom. Phys. \textbf{107}
  (2016), 149--161.

\bibitem{sparks_survey}
J.~Sparks, \emph{{Sasaki-{E}instein manifolds}}, Surv. Differ. Geom.
  \textbf{16} (2011), 265--324.

\bibitem{Vaisman_lcs}
I.~Vaisman, \emph{{Locally conformal symplectic manifolds}}, Internat. J. Math.
  Math. Sci. \textbf{8} (1985), no.~3, 521--536. \MR{809073}

\end{thebibliography}
\bibliographystyle{amsplain}
\end{document}